\documentclass[preprint,12pt]{elsarticle}
\usepackage{amsfonts,fixltx2e,amsmath,graphicx,amssymb,amsthm}
\usepackage{mathrsfs}
\usepackage{srcltx}
\usepackage{graphicx}
\usepackage{times}
\usepackage{amsfonts}
\usepackage{amsthm}
\usepackage{amsmath}
\usepackage{epic}
\usepackage{times}
\usepackage{multirow,array}
\usepackage{ulem}
\usepackage[margin=3cm]{geometry}
\usepackage{fullpage}

\usepackage{tikz}
\usetikzlibrary{snakes}
\usepackage{soul}

\date {}

\newtheorem{thm}{Theorem}
\newtheorem{theorem}{Theorem}[section]
\newtheorem{lemma}[theorem]{Lemma}
\newtheorem{corollary}[theorem]{Corollary}
\newtheorem{proposition}[theorem]{Proposition}
\newtheorem{example}[theorem]{Example}
\newtheorem{claim}{Claim}[theorem]
\theoremstyle{definition}
\newtheorem{definition}[theorem]{Definition}
\newtheorem{question}{Question}
\newtheorem{remark}[theorem]{Remark}

\newcommand{\nats}{{\mathbb N}}
\newcommand{\ints}{{\mathbb Z}}

\newcommand{\card}{{\rm Card}}
\def\A{\mathbb{A}}

\newcommand{\Fac}{{\rm Fac}}
\def\W{\mathcal{W}}
\def\L{\mathcal{L}}
\def\R{\mathcal{R}}

\newcommand{\todo}[1]{\marginpar{\tiny #1}}

\begin{document}

\begin{frontmatter}

\title{Cost and dimension of words of zero topological entropy}

\author[label1]{Julien Cassaigne}
  \ead{julien.cassaigne@math.cnrs.fr}

  \author[label2]{Anna E. Frid}
  \ead{anna.e.frid@gmail.com}

\author[label3]{Svetlana Puzynina\fnref{label5}}
  \ead{s.puzynina@gmail.com}

   \author[label4]{Luca Q. Zamboni}
  \ead{zamboni@math.univ-lyon1.fr}

\address[label1]{CNRS, Institut de Math\'ematiques de Marseille, France}
\address[label2]{Aix-Marseille Universit\'e, Institut de Math\'ematiques de Marseille, France}
\address[label3]{LIP, ENS de Lyon, Universit\'e de Lyon, France and Sobolev Institute of
Mathematics, Novosibirsk, Russia}
\address[label4]{Institut Camille Jordan, Universit\'e Lyon 1, France}
\fntext[label5]{Supported by the LABEX MILYON (ANR-10-LABX-0070) of Universit\'e de Lyon,
within the program ÒInvestissements d'AvenirÓ (ANR-11-IDEX-0007)
operated by the French National Research Agency (ANR).}

%\begin{document}
%\author{Julien Cassaigne, Anna E. Frid, Svetlana Puzynina, Luca Q. Zamboni}
%\title{Factor complexity and decompositions of the set of factors}

%\maketitle

\begin{abstract}
Let $\A^*$ denote the free monoid generated by a finite nonempty set $\A.$ In this paper we introduce a new measure of complexity of languages $L\subseteq \A^*$ defined in terms of the semigroup structure on $\A^*.$  For each  $L\subseteq \A^*,$ we define its {\it cost} $c(L)$ as the infimum of all real numbers $\alpha$ for which there exist a language $S\subseteq \A^*$ with $p_S(n)=O(n^\alpha)$ and a positive integer $k$ with  $L\subseteq S^k.$ We also define the {\it cost dimension} $d_c(L)$ as the infimum of the set of all  positive integers $k$ such that $L\subseteq S^k$ for some language $S$ 
with $p_S(n)=O(n^{c(L)}).$ We are primarily interested in languages $L$ given by the set of  factors of an infinite word $x=x_0x_1x_2\cdots \in \A^\nats$ of zero topological entropy, in which case $c(L)<+\infty.$ 
We establish the following  characterisation of  words of linear factor complexity: Let $x\in \A^\nats$ and $L=\Fac(x)$ be the set of factors of $x.$ Then $p_x(n)=\Theta(n)$ if and only $c(L)=0$ and $d_c(L)=2.$    In other words, $p_x(n)=O(n)$ if and only if $\Fac(x)\subseteq S^2$ for some language $S\subseteq \A^+$ of bounded complexity (meaning $\limsup p_S(n)<+\infty).$ 
In general the cost of a language $L$ reflects deeply the underlying combinatorial structure  induced by the semigroup structure on $\A^*.$ For example, in contrast to the above characterisation of languages generated by words of sub-linear complexity, there exist non factorial languages $L$ of  complexity $p_L(n)=O(\log n)$ (and hence of cost equal to $0)$ and of cost dimension $+\infty.$ 
In this paper we investigate the cost and cost dimension of languages defined by infinite words of zero topological entropy.  
We establish the existence of  words of cost zero and finite cost dimension having arbitrarily high polynomial complexity. In contrast we also show that for each $\alpha >2$ there exist infinite words $x$ of positive cost and of complexity $p_x(n)=O(n^\alpha).$   \end{abstract}

\journal{}

\end{frontmatter}

\section{Introduction}

Let $\A$ be a finite non-empty set. For each infinite word
$x=x_0x_1x_2\cdots \in\A^\nats$, the complexity or {\it factor
complexity} $p_x(n)$  counts the number of distinct blocks
$x_ix_{i+1}\cdots x_{i+n-1}\in \A^n$ of length $n$  occurring in
$x$.  In other words, the complexity of $x$ is taken to be the
complexity of the language of its factors
$\Fac(x)=\{x_ix_{i+1}\cdots x_{j}\,|\, 0\leq i\leq j\}$. First
introduced by Hedlund and Morse in their seminal 1938 paper
\cite{MoHe1} under the name of {\it block growth},\footnote{In
\cite{elr}, Ehrenfeucht, Lee, and Rozenberg adopted the term
\it{subword complexity.}} the factor complexity provides a
useful measure of the extent of randomness of  $x$ and more
generally of the subshift it generates.  Periodic words have
bounded factor complexity while digit expansions of normal numbers
have maximal complexity. A celebrated theorem of Morse and Hedlund
in \cite{MoHe1} states that every aperiodic (meaning non-ultimately periodic) word contains at
least $n+1$ distinct factors of each length $n.$ 
%Words satisfying $p_x(n)=n+1$ for each $n\geq 0$ are called Sturmian words, and are regarded, in just about every respect, as the simplest aperiodic words.
 Results on the complexity of words are generally one of  two kinds: Either they provide conditions or formulae for the complexity of a given family of words, for instance Pansiot's work in  \cite{pansiot} on the classification of the factor complexities of  morphic words.  Or they give conditions on words, or rules for generating them, subject to specified constraints on their complexity. An example of a deep and difficult  problem of this kind is the so-called $S$-adic conjecture on words of linear complexity (see for instance \cite{leroy} and the references therein).
 
 The set $\A^*$ consisting of all finite words over the alphabet $\A$ is naturally a free monoid  under the operation of concatenation, with the empty word $\varepsilon$ playing the role of the identity.  Thus given a language $L\subseteq \A^*$ (for instance consisting of all factors of some infinite word $x\in \A^\nats)$ one may ask whether $L$ is contained in a finite product of the form $S^k$ where $S$ is a language of strictly lower complexity. Consider for example the {\it Thue-Morse infinite word}  \[x=011010011001011010010\cdots \]
where for each $n\geq 0,$ the $n'$th term $x_n$ is defined as the sum modulo $2$ of the digits  in the binary expansion of
$n$. The origins of this word date back to the beginning of the last century with the works of  A. Thue \cite{Th1, Th2} in which he proves amongst other things that $x$ is {\it overlap-free} i.e., contains no word of the form $uuu'$ where $u'$ is a non-empty prefix of $u.$ It is well known that $x$ is also a fixed point of the substitution $\varphi: 0\to 01, 1 \to 10$. The factor complexity of the Thue-Morse word, first computed by Brlek \cite{Br} and independently by de Luca and  Varricchio \cite{dLV}, is given by $ p_x(1)=2, p_x(2)=4$ and for $n\geq 3$
\[p_x(n)=\begin{cases} 6\cdot 2^{r-1}+4q\,\,\,\,\,\,\,\,\,\,\,\,0< q\leq 2^{r-1}\\ 
2^{r+2}+2q\,\,\,\,\,\,\,\,\,\,\,\,\,\,\,\,\,\,\,2^{r-1}<q\leq2^r
\end{cases}\]
where $r$ and $q$ are uniquely determined by the equation $n=2^r+q+1,$ $r\geq 0$ and $0< q \leq2^r.$  
For each $n\geq 0$, let $t_n=\varphi^n(0)$ and $\overline{t}_n=\varphi^n(1)$. Then both $t_n$ and  $\overline{t}_n$ are factors of $x$ of length $2^n$.
Let $S\subseteq \{0,1\}^*$ be the set consisting of all  prefixes and suffixes (including $\varepsilon$) of $t_n$ and $\overline{t}_n$ for each $n\geq 0$. Since $t_{n+1}=\varphi^{n+1}(0)=\varphi^n(01)=t_n\overline{t}_n$ and similarly $\overline{t}_{n+1}=\overline{t}_nt_n$,  it follows that  $S$ contains at most $4$ words of each length $n$. We claim that $\Fac(x)\subseteq S^2.$
To see this, let $u\in \Fac(x)$. Since $S$ contains $\varepsilon, 0$ and $1$, we may suppose $|u|\geq 2$. Consider the least $n\geq 0$ such that $u$ is a factor of $t_{n+1}$ or a factor of $\overline{t}_{n+1}$.  If $u$ is a factor of $t_{n+1}$, by minimality of $n$ we can write $u=vw$ where $v$ is a non empty suffix of $t_n$ and $w$ a non empty prefix of $\overline{t}_n$. Whence $u\in S^2$. A similar argument applies in case $u$ is a factor of $\overline{t}_{n+1}.$ Thus while $\Fac(x)$ is of linear complexity, it is contained in a product $S^2$ where $S$ is a language of bounded complexity. With some care, this construction may be generalized to fixed points of arbitrary (primitive) substitutions $\tau: \A\rightarrow \A^+$  by letting $S$ be the collection of all prefixes and suffixes of all $\tau^n(u)$ $(n\geq 0)$ where $u$ ranges over all factors of $\tau(a)$ for each  $a\in \A$.  

As another example, let $L=\Fac(x)$ denote the set of  factors of a Sturmian word $x\in \{0,1\}^\nats$ (for instance we may take $x= 0100101001001010010\cdots$ the Fibonacci word defined as the fixed point of the substitution $0 \mapsto 01, 1\mapsto 0).$   We recall that Sturmian words are infinite words having exactly $n+1$ distinct factors of each length $n\geq 1.$ In view of the Morse-Hendlund theorem, Sturmian words are those aperiodic words of minimal factor complexity. They arise naturally in various branches of mathematics  including combinatorics, algebra, number theory, ergodic theory, dynamical systems and differential equations. In theoretical physics, Sturmian words constitute $1$-dimensional models for quasi-crystals, and in theoretical computer science they are used in computer graphics as digital approximation of straight lines.  %Sturmian words admit many different characterisations of combinatorial, geometric and arithmetic nature (see for instance Chapter 2 of \cite{Lo}).  
The condition $p_x(n)=n+1$ implies that $x$ admits a unique left (right, respectively) special factor of each length $n$ denoted $l_x(n)$  ($r_x(n),$ respectively). In other words, $l_x(n)$ ($r_x(n),$ respectively) occurs in $x$ preceded (followed, respectively) by both $0$ and $1.$ See for instance Chapter 2 of \cite{Lo}.  Set $S=\{\varepsilon\}\cup\{r_x(n)0\,|\, n\geq 0\}\cup\{1l_x(n)\,|\, n\geq 0\}.$ Then $S$ consists of precisely  $2$ words of each given length $n\geq 1.$  One can prove that $\Fac(x) \subseteq S^2$ (see Example~\ref{e:st}).    It turns out that this is optimal in the sense that if $x$ is an infinite word and $\Fac(x) \subseteq S^2$ for some language $S$ with $\limsup p_S(n)\leq 1,$ then $x$ is ultimately periodic (see \cite{Za}). \\

The above examples are only special cases of the following more general result proved herein: 

\begin{thm}\label{Prim} An infinite word $x\in \A^\nats$ is of sub-linear complexity (i.e., $p_x(n)=O(n))$ if and only if $\Fac(x)\subseteq S^2$ for some language $S\subseteq \A^*$ of bounded complexity (i.e., $\limsup p_S(n)<+\infty).$
\end{thm}

%The essential idea is that a given a language $L$ it is sometimes possible to find a language $S$ of strictly lower complexity than that of $L,$ with $L\subseteq S^k$ where $k$ is some fixed positive integer. %We are especially interested in these and related  questions in the case in which $L$ is the set of factors of some infinite word of zero topological entropy. % $\lim_{n\rightarrow \infty}\frac{\log p_x(n)}{n}=0.$ 

%In other words, the language of factors of an infinite word of sub-linear complexity is contained in a product $S^2$ for some language $S$ of bounded complexity. 
It turns out that Theorem~\ref{Prim} is very specific to languages defined by infinite words. In fact, there exist non-factorial languages $L$ of complexity $p_L(n)=O(\log n)$ which are not contained in any finite product of the form $S^k$ where $S$ is a language of bounded complexity and $k$ a positive integer.

Our aim here is to express and study these ideas in greater generality. Given a language $L$ of  low complexity, meaning $\limsup_{n\rightarrow \infty} \frac{\log p_L(n)}{n}=0,$   we define the {\it cost} of $L,$ denoted $c(L),$ as the infimum of all real numbers $\alpha$ for which there exist a language $S$ with $p_S(n)=O(n^\alpha)$ and a positive integer $k$ such that $L\subseteq S^k.$ More precisely, for each real number $\alpha \in [0,+\infty)$, we define the $\alpha$-{\it dimension} $d_\alpha (L) $ by
\[d_\alpha(L)=\inf \{k\geq 1\,|\, L\subseteq S^k\,\,\mbox{for some language}\,\, S\in \L(\alpha)\},\]
where $\L(\alpha)$ denotes the collection of all languages
$L\subseteq \A^*$ (over some finite non empty alphabet $\A)$ whose
complexity $p_L(n)=O(n^\alpha)$. %If $L\in \L(0)$, then the number of words of $L$ of each length $n$ is uniformly bounded,  and we say $L$ is of {\it bounded complexity}. %Analogously, let $\W(\alpha)$ be the collection of all infinite words $x\in \A^\nats$ (over some finite non empty alphabet $\A)$ such that $\Fac(x)\in \L(\alpha)$.
If $x$ is an infinite word and $L=\Fac(x),$ then, by the Morse-Hedlund theorem, $L$ belongs to $\L(0)$ if and only if $x$ is ultimately periodic. While if $x$ is a Sturmian word or if $x$ is generated by a primitive substitution, then $L$ belongs to $\L(1).$
Then the  cost $c(L)$  is given by
\[c(L)=\inf \{\alpha \in [0,+\infty)\,|\, d_\alpha(L) <+\infty\}.\]
In each case above we take the convention that $\inf \emptyset =+\infty$. If $c=c(L)<+\infty$, then we call $d_c(L)\in \{1,2,3,\ldots \}\cup \{\infty\}$  the {\it cost dimension} of $L$. In case $L=\Fac(x)$ for some infinite word $x,$ then we write $c (x)$ ($d_c(x),$ respectively) in lieu of $c(L)$ ($d_c(L),$ respectively). Thus, the Morse-Hedlund theorem states that an infinite word $x\in \A^\nats$ is ultimately periodic if and only if $c(x)=0$ and $d_0(x)=1,$ i.e., $x$ is of cost equal to $0$ and cost dimension equal to $1.$ Similarly, Theorem~\ref{Prim} asserts that $x$ is of linear complexity (i.e., $p_x(n)=\Theta(n))$ if and only if $x$ is of cost equal to $0$ and cost dimension equal to $2.$
The above definitions  may be adapted to other measures of complexity as we  do herein for  the so-called {\it accumulative complexity} $p^*_L(n)$ which counts the number of words in $L$ of length less than or equal to $n$.

%This leads to the notions of {\it cost} and {\it cost dimension} defined as follows: For each real number $\alpha \in [0,+\infty)$, let $\L(\alpha)$ denote the collection of all languages $L\subseteq \A^*$ (over some finite non empty alphabet $\A)$ whose complexity $p_L(n)=O(n^\alpha)$. If $L\in \L(0)$, then the number of words of $L$ of each length $n$ is uniformly bounded,  and we say $L$ is of {\it bounded complexity}. %Analogously, let $\W(\alpha)$ be the collection of all infinite words $x\in \A^\nats$ (over some finite non empty alphabet $\A)$ such that $\Fac(x)\in \L(\alpha)$. If $x$ is an infinite word and $L=\Fac(x),$ then, by the Morse-Hedlund theorem, $L$ belongs to $\L(0)$ if and only if $x$ is ultimately periodic. While if $x$ is a Sturmian word or if $x$ is generated by a primitive substitution, then $L$ belongs to $\L(1).$

%\medskip
%\noindent {\bf Definition. }{\rm Let $L\subseteq \A^*$. For each real number $\alpha \in [0,+\infty)$, we define the $\alpha$-{\it dimension} $d_\alpha (L) $ by \[d_\alpha(L)=\inf \{k\geq 1\,|\, L\subseteq S^k\,\,\mbox{for some language}\,\, S\in \L(\alpha)\},\] and the {\it cost} $c(L)$  by \[c(L)=\inf \{\alpha \in [0,+\infty)\,|\, d_\alpha(L) <+\infty\}.\]}

\medskip
\noindent 
%Thus languages of positive entropy have cost equal to $+\infty$.

%We note that if  $L\in \L(\alpha_0)$ for some $\alpha_0\geq 0$, then $d_\alpha(L)=1$ for each $\alpha\geq \alpha_0$ and hence $c(L)\leq \alpha_0$. However as we saw in the earlier  examples,  the cost of a language $L$ may  be strictly smaller than its complexity $\alpha_0.$ For instance, if $x$ is a Sturmian word, or if $x$ is the Thue-Morse word, then $\Fac(x)\subseteq S^2$ where $S$ is a language of bounded complexity. Whence $c(x)=0$ and $d_0(x)=2.$ 

A fundamental question, to which a substantial portion of the paper is devoted, is to what extent does the complexity of a language 
determine its cost and cost dimension and vice versa.  A first basic observation is that languages $L$ of positive entropy have cost equal to $+\infty.$ For this reason we restrict our attention to languages and words of zero topological entropy. Via a straightforward  counting argument, it is shown that for each $\alpha \geq 0,$ if $d_\alpha(L)=k$ for some $1\leq k<+\infty$, then $L\in \L(k(\alpha +1)-1).$
It follows from this that   $c(L)$ is finite  if and only if the complexity of $L$ is bounded above by a polynomial.
%In case $x$ is an infinite word, then taking $L=\Fac(x),$  $\alpha =0$ and $k=1$ we deduce (thanks to the Morse-Hedland theorem) that  $x$ is ultimately periodic  if and only if $d_0(x)=1.$ 
We further show by direct construction that  for each positive integer $k\geq 1$ there exists an infinite word $x$ of complexity $p_x(n) \in \Omega(n^{k-1})$ with $d_0(x)=k.$   In 
other words, we establish the  existence of words of cost zero and of arbitrarily high polynomial complexity.  

Conversely, given the complexity of a language, what can be said of its cost and cost dimension. We already mentioned two results in this direction: first the obvious fact that for languages $L$ of bounded complexity we have $d_0(L)=1.$ Second, that if $L$ is the set of factors of an aperiodic infinite word, then $L$ is of linear complexity if and only if its cost $c(L)=0$ and its cost dimension $d_0(L)=2.$  However in general, the cost and cost dimension of a given language depend only in part on its complexity.  In fact, both reflect  deeply the underlying combinatorial  structure of the language.  For instance, we already mentioned that non-factorial languages are in general very far from satisfying any result along the lines of Theorem~\ref{Prim}. 
%$L$ of complexity $p_L(n)=O(\log n)$ for which $d_0(L)=+\infty.$ This should be contrasted with the fact that there exist languages $L$ of complexity $p_L(n)= \Omega(n^{k-1})$ and $d_0(x)=k.$ 
%However, for languages $L$ defined by infinite words, the answer is more satisfactory, at least in the case of very low complexities. One of the main results of this paper is the following characterisation of words of linear complexity:  %(see Theorem~\ref{c:main}):
%\medskip \noindent {\bf Theorem. } Let $x\in \A^\nats$. Then $p_x(n)=\Theta(n)$ if and only if $d_0(x)= 2$ (i.e., the cost of $x$ equals 0 and its cost dimension equals $2).$   \medskip \noindent
But even in the case of languages defined by infinite words, the characterisation of Theorem~\ref{Prim} does not seem to extend nicely to higher complexities. For instance, we prove that the word $x=\prod_{i=1}^{\infty} ab^i=ababbabbb\cdots$
generated by the (non-primitive) substitution $a \mapsto ab, b \mapsto b, c \mapsto ca$, considered by Pansiot in \cite{pansiot} and of complexity $p_x(n)=\Theta(n^2),$ verifies $d_0(u)>3.$ On the other hand we also show that $d_0(x)\leq 6$ which in particular implies is of cost zero. We do not know whether there exist words of sub-quadratic complexity and positive cost.  However,  we  prove  that  for every real number $\alpha \in (0,1)$ there exists an infinite word $x$ with complexity $p_x(n)\in O(n^{2+\alpha})$ and  cost $c(x)\geq \alpha.$ In other words,  there exist words of positive cost having relatively low (sub-cubic) complexity. This should be contrasted with 
the result mentioned earlier on the existence of words of arbitrarily high polynomial complexity having cost equal to zero.
These results suggest that the cost of a word measures something  beyond its factor complexity which makes it of independent interest. 

The paper is structured as follows: In \S2 we briefly recall some of the basic terminology and notions arising in the study of infinite words. For a more detailed exposition, the reader is referred to one of the standard texts in combinatorics on words such as the Lothaire books \cite{LoComb, Lo, LoApp}.  Also in \S2, for the sake of clarity and self-containment, we develop in detail some notions which are less mainstream in the area of combinatorics on words and yet relevant in what follows, in particular used in the proofs of the main results. They include the notions of internal and extremal occurrences of factors in both finite and infinite words which are defined in terms of virtual occurrences and local periods.   In \S3 we define the key notions of cost and cost dimension of a language in the context of  the factor complexity as well as  the accumulative complexity. Also in this section we establish various fundamental results linking the cost of a language to its complexity and relations between the cost $c(L)$  defined in terms of the factor complexity and the cost $c^*(L)$ defined in terms of the accumulative complexity.  In \S4 we  study the cost and cost dimension of words of sub-linear complexity. 
We begin \S4 by introducing the notions of {\it marker words} and {\it marker sets} which are both new and may be of independent interest. Marker sets defined by right special factors constitute the key tool needed to split each factor of an infinite word of linear complexity into two pieces. 
This decomposition enables us to obtain what we regard to be the main result of the paper (see Theorem~\ref{c:main}), and which gives a complete characterisation of words of linear complexity in terms of cost and cost dimension: An infinite word $x$ is of linear complexity, i.e., $p_x(n)=\Theta(n)$ if and only if the cost $c(x)=0$ and the cost dimension $d_0(x)=2$.  Theorem~\ref{c:main}  is actually a consequence of a more general result given by Theorem~\ref{t:main} combined with an earlier result of the first author which gives a uniform bound on the number of right special factors of each length $n$ of an infinite word word of linear complexity.    
In \S5 we study the cost and cost dimension of words of sub-quadratic complexity. We begin the section with another consequence of Theorem~\ref{t:main} which yields a non-trivial bound on the cost of words $x$ of complexity $p_x(n)=O(n^\alpha)$ for $\alpha \in (1,2).$
We estimate the cost complexity of the fixed point $x$ of the substitution $a \mapsto ab, b \mapsto b, c \mapsto ca$ which is known to have quadratic complexity and prove that $4\leq d_0(x)\leq 6.$ In particular this shows that the result of Theorem~\ref{c:main} already breaks down for words of quadratic complexity.  In \S6 we investigate the cost and cost dimension of words of greater than quadratic complexity and prove that every real number $\alpha \in (0,1)$ there exists an infinite word $x$ with complexity $p_x(n)\in O(n^{2+\alpha})$ and  cost $c(x)\geq \alpha$ (see Corollary~\ref{c:alpha}). Finally in \S7 we exhibit an example of a non-factorial language $L$ of complexity $p_L(n)=O(\log n)$ (and hence of cost zero) having infinite cost dimension i.e.,  $d_0(L)=+\infty.$ This is yet another illustration of  how the main result of Theorem~\ref{c:main} depends strongly on the assumption that the language $L$ be defined by an infinite word.

\section{Preliminaries}

In this section we briefly recall some basic definitions and notations concerning
finite and infinite words which are relevant to the subsequent
sections. For more details we refer the reader to \cite{Lo}. We also introduce the new notions of internal and extremal occurrences of factors in finite and infinite words which are defined by their virtual occurrences and local periods. 

Let
$\A$ be a finite non-empty set (the {\it alphabet}). Let $\A^*$
denote the set of all finite  words $u=u_0u_1\cdots u_{n-1}$ with
$u_i\in \A$.   We call $n$ the {\it length} of $u$ and denote it
$|u|$. The empty word is denoted $\varepsilon$ and by convention
$|\varepsilon|=0$. We put $\A^+=\A^*\setminus \{\varepsilon\}$.
For each $u\in \A^*$ and $a\in \A$, we let $|u|_a$ denote the
number of occurrences of $a$ in $u$. For $u=u_0u_1\cdots u_{n-1}\in \A^+$ we define  
\[\Fac(u)=\{u_i\cdots u_{j}:\, 0\leq i\leq j \leq n-1\}\cup \{\varepsilon\}.\]
%\xout{The prefix of $u$ of length $\pi(u)$  is
%called the {\it period} of $u$.}\todo{never used}
A subset $L\subseteq \A^*$ is called a {\it language}. A language $L$ is said to be {\it factorial} if $\Fac(u)\subseteq L$ for each $u\in L$.  Given a  language $L\subseteq \A^*$, we define its complexity $p_L:\nats \rightarrow \nats$ by
\[p_L(n)=\card (L\cap \A^n)\]
and its accumulative complexity $p^*_L:\nats \rightarrow \nats$ by

\[p^*_{L}(n)=\sum_{i=0}^{n} p_{L}(i).\]

Let $\A^\nats$  denote the set of all right infinite words $x=x_0x_1x_2\cdots$  with $x_i\in \A$.
Given $x=x_0x_1x_2\cdots \in \A^*\cup \A^\nats$   let $\Fac(x)=\{x_i\cdots x_{i+n}:\, i,n \geq 0\}\cup \{\varepsilon\}$ denote the set of factors if $x.$  We will frequently use the notation $x[i,j] $ for $x_i\cdots x_j$.  A factor $u$ of $x$ is called {\it right} (resp., {\it left}) {\it special} if $ua,ub \in \Fac(x)$
(resp., $au,bu \in\Fac(x)$) for distinct letters $a,b
\in \A$.  Let $p_x:\nats \rightarrow \nats$ (resp., $p^*_x:\nats \rightarrow \nats)$ denote the {\it factor complexity} (resp., {\it accumulative factor complexity}) of $x$ defined by:
 \[p_x(n)=\card (\Fac(x)\cap \A^n)\]
 and
 \[p^*_{x}(n)=\sum_{i=0}^{n} p_{x}(i).\]
 We say $x\in \A^\nats$ (resp., $L\subseteq \A^*)$ is of {\it bounded complexity} if there exists a positive integer $C$ such that $p_x(n)\leq C$ (resp., $p_L(n)\leq C)$ for all $n\in \nats$.
%\xout{A factor $u$ of $x$ which is both right and left special is called {\it bispecial}.}\todo{never really used} 
An infinite word $x$ is called {\it ultimately periodic}, or {\it ultimately} $|v|$-{\it  periodic},  if $x=uvvv\cdots = uv^{\omega}$ for some non-empty words $u,v\in \A^*$. An infinite word is said to be {\it aperiodic} if it is not ultimately periodic.
It follows that every aperiodic word contains a right and a left special factor of each length.
An infinite word $x$ is said to be {\it recurrent} if each prefix of $x$ occurs infinitely often in $x.$ 

Analogously we can consider  bi-infinite words indexed by $\ints$. The definitions above extend in the obvious ways. In particular, a bi-infinite word $x$ is said to be eventually periodic if it is eventually periodic to both the left and the right, i.e., if $x$ admits a prefix of the form $\cdots uuu$ and a suffix of the form $vvv\cdots $ for some $u,v\in \A^+.$ Otherwise $x$ is said to be aperiodic.

 \begin{definition} Let $u=u_1u_2\cdots u_n$ and $v$ belong to $ \A^+$ and fix $1\leq i\leq n.$  We say there there is a {\it virtual occurrence} of $v$ in $u$ beginning (ending, respectively) at position $i$ if the shorter of $v$ and $u_i\cdots u_n$ ($u_1\cdots u_{i-1}, $ respectively) is a prefix (suffix, respectively) of the other. That is $v\A^*\cap u[i,n]\A^* \neq \emptyset$ $(\A^*v \cap \A^*u[1,i-1]\neq \emptyset, $ respectively). 
 \end{definition} 
 
 \begin{definition} For $u=u_1u_2\cdots u_n$ and $1\leq i\leq n,$ we say that $u$ has a {\it virtual square} 
 centered at  position $i$ if there exists a word $v\in \A^+$ (the {\it witness}) and a virtual occurrence of $v$ in $u$ both beginning and ending at position $i.$
 \end{definition}
 
 For example, the word $u=00101101$ has a virtual square of length $2$ at position $i=3$ (witnessed by $v=01)$ as well as a virtual square of length $3$ at position $i=7$ (witnessed by $v=110.)$ 
 
The above definitions extend in the obvious way to define a virtual occurrence of a word $v\in \A^+$ beginning or ending at a position $i\geq 0$ in an infinite word $x=x_0x_1\cdots .$  In this way we can talk about virtual squares occurring in an infinite word. 
%If a factor $v$ of length $n$ occurs in a finite or infinite indexed word $u$ starting from position $i$, that is, $v=u_i\cdots u_{i+n-1}$, we denote this fact by $v=u[i,i+n-1]$ and speak of an occurrence of $v$ to $u$. A one-sided infinite word is called {\it recurrent} if any factor appears in it at least twice. This definition is equivalent to the following one which we will need: a word is recurrent if and only if any its factor appears in it as far from the beginning as we need.
%Given a finite word $u=u_0u_1\cdots u_{n-1}\in \A^+$ and a positive integer $m$, we say $u$ has a {\it virtual square} of length $m$ at position $i$ (with $0\leq i\leq n)$ if there exists a word $z$ of length $m$ such that $\A^*z\cap \A^* u[0,i-1]\neq \emptyset$ and $z\A^* \cap u[i,n]\A^*\neq \emptyset$. Similarly, given an infinite word  $x=x_0x_1x_2\cdots \in \A^\nats$ and a positive integer $m$, we say $x$ has a  virtual square of length $m$ at position $i\geq 0$ if  $\A^*x[i,i+m-1]\cap \A^*x[0,i-1]\neq \emptyset$, i.e., if the shorter of $x_0x_1\cdots x_{i-1}$ and $x_ix_{i+1}\cdots x_{i+m-1}$ is a suffix of the other. 
For instance, the word $x=0100101001001010010\cdots$ has
virtual squares of length $2$ and $3$ at position $1$, and of
lengths $3$ and $5$ at position $2$.

\begin{definition} For $v=v_1 v_2 \cdots v_n \in \A^+. $ Define the (least) {\it period} of $v,$ denoted $\pi(v),$ to be the least positive integer $m$ such that $v_i=u_{i+m}$ for all $1\leq i\leq n-m$. 
\end{definition} 

For instance, for $v=00110$ we have $\pi(v)=4$ while for $v=00101101$ we have $\pi(v)=8=|v|.$  
Clearly in general $\pi(v)\leq |v|$.

Let $x\in \A^+\cup \A^\nats$ be a finite or
infinite word, and let $v\in \A^+$ be a word occurring in
$x$ at a position $i\geq 0,$ meaning $v=x[i,i+n-1]$. We say that the
occurrence of $v$ at position $i$ is {\it internal} if $x$ has a
virtual square of length $\pi(v)$ centered at positions $i$ and $i+n.$  An occurrence of $v$ in $x$ which is not internal is called
{\it extremal}. More precisely, an extremal occurrence is called
{\it initial} if $x$ does not have a virtual square of length
$\pi(v)$ centered at position $i$, and {\it final} if $x$ does not have
virtual square of length $\pi(v)$ at position $i+n$. For instance,
if $x=01001010100\cdots$, then the occurrence of $v=010$ at
position $0$ is not initial since $x$ has a virtual square of
length $2=\pi(v)$  centered at position $0$. Instead this
occurrence is final (even if it is immediately followed by another
occurrence of $v)$  since $x$ does not a virtual square of length
$2$ centered at position $3$. On the other hand, the occurrence of
$v$ at position $3$ is initial since $x$ does not have a virtual
square of length $2=\pi(v)$ centered at position $3$. In contrast,
the occurrence of $v$ in position $5$ is internal. Note that an
occurrence of a word $v$  in $x$ can be both initial and final. We
also note that if  $x$ is aperiodic, then each factor $v$ of $x$
admits a final occurrence in $x$.

%S: Notations from school books on analysis 1, I strongly believe we should remove. % 
Throughout the paper we make use of the usual Landau notations $O, \Omega,\Theta, $ and $o.$ We adopt the following definition of $\Omega$ which is more commonly used in computer science: 
Given functions $f,g:\nats \rightarrow \mathbb R^+ $, we write
\begin{align*}
%f(n)=O(g(n)) &\mbox{~~if~~~~}\exists K>0,\, \exists N, \,\forall n\geq N: \,f(n)\leq Kg(n);\\
f(n)=\Omega(g(n)) &\mbox{~~if~~~~} \exists K>0,\, \exists N,\, \forall n\geq N:\, f(n)\geq Kg(n).\\
%f(n)=\Theta(g(n)) &\mbox{~~if~~~~} \exists K_1>0,\, \exists K_2>0,\,  \exists N, \,\forall n\geq N: \,K_1g(n)\leq f(n)\leq K_2g(n);\\
%f(n)=o(g(n)) &\mbox{~~if~~~~}\forall K>0,\, \exists N,\, \forall n\geq N:\, f(n)\leq Kg(n).
\end{align*}

 \section{Dimension and cost: definitions, examples and general properties}\label{s:examples}

%For the sake of clarity, we repeat here some definitions from the Introduction.

For each real number $\alpha \in [0,+\infty)$, we denote by $\L(\alpha)$ (resp., $\L^*(\alpha))$ the collection of languages $L\subseteq \A^*$
 (over some finite non empty alphabet $\A)$ with $p_L(n)=O(n^\alpha)$ (resp., $p^*_L(n)=O(n^\alpha))$. Analogously, we denote by
 $\W(\alpha)$ (resp., $\W^*(\alpha))$ the collection of  infinite words $x\in \A^\nats$ (over some finite non empty alphabet $\A)$ such that $\Fac(x)\in \L(\alpha)$ (resp., $\Fac(x)\in \L^*(\alpha))$. The set $\A^*$ is considered as a free monoid, and thus for each $S \subseteq \A^*$ the set $S^k$ is just the set of all concatenations of $k$ elements of $S$.

\begin{definition}\label{new}{\rm Let $L\subseteq \A^*$. For each real number $\alpha \in [0,+\infty)$, we define the $\alpha$-{\it dimension} $d_\alpha (L) $ by
\[d_\alpha(L)=\inf \{k\geq 1\,|\, L\subseteq S^k\,\,\mbox{for some language}\,\, S\in \L(\alpha)\},\]
and the {\it cost} $c(L)$  by
\[c(L)=\inf \{\alpha \in [0,+\infty)\,|\, d_\alpha(L) <+\infty\}.\]
If $c=c(L)<+\infty$, we call $d_c(L)\in [1,+\infty]$ the {\it cost dimension} of $L$.}
\end{definition}

%\todo{Discuss difference between $\subseteq S^k$ and $=S^k$ and $=S_1S_2\cdots S_k$.}

\noindent By convention $\inf \emptyset =+\infty$. %It follows immediately from the definition that languages of positive entropy have cost equal to $+\infty$.
Definition~\ref{new} extends naturally to infinite words $x\in \A^\nats$  by replacing  $L$ by $\Fac(x)$ so we define accordingly  $d_\alpha (x)$ and $c(x).$ Replacing $\L(\alpha)$ by $\L^*(\alpha)$ we define analogously the  $\alpha$-{\it accumulative dimension} $d^*_\alpha(L)$ and the {\it accumulative cost} $c^*(L)$.

We observe that in our definition of $d_{\alpha}(L)$, we may replace $S^k$ by $S_1\cdots S_k$ for some languages $S_1,\ldots,S_k \in \L(\alpha).$ %Indeed, to pass from $S^k$ to $S_1\cdots S_k$ we just take $S_i=S$ for all $i$; and to move in the opposite direction, we take $S=\cup_{i=1}^k S_i$.
%By contrary, the sign $\subseteq$ in the definition of $d_{\alpha}(L)$ cannot be replaced by $=$. Equality is much more restrictive as a condition than inclusion, and just leads to another kind of problems. 
The following lemma is an immediate consequence of the definition:

\begin{lemma}\label{cdim} Suppose $L\in \L(\alpha_0)$ (resp., $L\in \L^*(\alpha_0))$ for some $\alpha_0\geq 0$. Then
$d_\alpha(L)=1$ (resp., $d^*_\alpha(L)=1)$ for each $\alpha\geq \alpha_0$ and hence
$c(L)\leq \alpha_0$ (resp., $c^*(L)\leq \alpha_0)$.
\end{lemma}

\begin{lemma}\label{up} For each language $L\subseteq \A^*$,  we have $d_0(L)=1$ if and only if $L$ is of bounded complexity. For each infinite word $x\in \A^\nats$, we have $d_0(x)=1$ if and only if $x$ is ultimately periodic.
\end{lemma}

\begin{proof}The first statement is clear from Definition~\ref{new}. As for the second, if $x$ is ultimately periodic, then its complexity is bounded, whence $d_0(x)=1$. Conversely if $d_0(x)=1$, then the complexity of $x$ is bounded, and hence by the Morse-Hedlund theorem $x$ is ultimately periodic. \end{proof}

\begin{example}[Sturmian words]\label{e:st}  {\rm Here we prove that for every Sturmian word $x$ we have $d_0(x)=2.$  To see this, we show that for each Sturmian word $x\in \{0,1\}^\nats$, there exist sets $S,T$ with $p_S(n),p_T(n) \equiv 1$ (for each $n\geq 0)$ such that $\Fac(x)\subseteq ST.$  Combined with Lemma~\ref{up}, this implies that $d_0(x)=2.$ The condition $p_x(n)=n+1$ implies that $x$ admits a unique left (right, respectively) special factor of each length $n$ denoted $l_x(n)$  ($r_x(n),$ respectively).  Moreover, as is well known,  $l_x(n)$ and $r_x(n)$ are reversals of one another.  
Set $S=\{\varepsilon\}\cup\{r_x(n)0\,|\, n\geq 0\}$ and $ T=\{\varepsilon\}\cup\{1l_x(n)\,|\, n\geq 0\}.$  Then clearly, $p_S(n),p_T(n) \equiv 1.$
It remains to show that $\Fac(x)\subseteq ST.$  To this end we recall that for each $n\geq 1$, the word $w(n)=r_x(n-1)01l_x(n-1)$ is a factor of $x$ of length $2n$ (see for instance Exercise 6.1.24 in \cite{Arn}).  We claim that for each $n\geq 1$, $w(n)$ contains $n+1$ distinct factors of length $n$. Assuming for a moment this claim, it follows that each factor of $x$ of length $n$ is a factor of $w(n)$ and hence $\Fac(x)\subseteq ST$ as required. To prove the claim, we proceed by induction on $n$. For $n=1$, we have $w(1)=01$ which contains $2$ factors of length $1$. For the inductive step, let $n\geq 1$, and assume $w(n)$ contains $n+1$ distinct factors of length $n$. We wish to show that $w(n+1)$ contains $n+2$ distinct factors of length $n+1$. Suppose to the contrary that some word $u$ of length $n+1$ occurs twice in $w(n+1)$. We claim $u= r_x(n)0$, for otherwise the word $u'$ obtained by deleting the  last letter of $u$ would occur twice in $w(n)$, a contradiction.
 Similarly, if $u\neq 1l_x(n)$, then the word $u''$ obtained by deleting the first letter of $u$ would occur twice in $w(n)$, a contradiction. Thus $u=r_x(n)1=0l_x(n)$, which is impossible since, as $r_x(n)$ and $l_x(n)$ are reversals of one another, we have that $r_x(n)1$ and $0l_x(n)$ do not contain the same number of $0'$s and $1'$s. } \end{example}

The next proposition illustrates the basic relations between the dimension $d_\alpha$ and the accumulative dimension $d^*_\alpha$. It is stated in terms of languages $L\subseteq \A^*$ but the same inequalities hold for infinite words $x\in \A^\nats$.

\begin{proposition}\label{t:l1l2}For each $\alpha \geq 0$ and language $L\subseteq \A^*$ we have
\begin{enumerate}
\item$d_\alpha (L)\leq d^*_\alpha (L)$,
\item $d^*_{\alpha +1}(L) \leq d_\alpha (L)\leq 2d^*_{\alpha+1}(L)$.
\end{enumerate}
\end{proposition}

\begin{proof}We begin by showing that $d_\alpha (L)\leq d^*_\alpha (L)$. The result is clear if $d^*_\alpha(L)=+\infty$. Thus assume $d^*_\alpha(L)=k$ for some positive integer $k$. Then $L\subseteq S^k$ for some language $S\in \L^*(\alpha)$. Hence $S\in \L(\alpha)$ whence $d_\alpha(L)\leq k=d^*_\alpha(L)$ as required.
Next we show that $d^*_{\alpha +1}(L) \leq d_\alpha (L)$. Again the result is clear if $d_\alpha(L)=+\infty$, thus we may suppose $d_\alpha(L)=k$ for some positive integer $k$. Then $L\subseteq S^k$ for some language $S\in \L(\alpha)$. In other words, $p_S(n)=O(n^\alpha)$. Thus $p^*_S(n)=O(n^{\alpha+1})$, i.e., $S\in \L^*(\alpha+1)$, and hence $d^*_{\alpha +1}(L)\leq k= d_\alpha(L)$.  In order to prove the remaining inequality, we will need the following lemma:

\begin{lemma}\label{ST}Let $T\subseteq \A^*$. If $T\in \L^*(\alpha +1)$, then $T\subseteq S^2$ for some $S\in \L(\alpha)$.
\end{lemma}

\begin{proof}Since $T\in \L^*(\alpha +1)$, there exists a constant $K>0$ such that $p^*_T(n)\leq Kn^{\alpha +1}$ for each $n\geq 1$. We order $T=\{v_1,v_2,v_3,\ldots\}$ so that $|v_m|\leq |v_{m+1}|$ for each $m\geq 1$.
Thus for each $m\geq 2$ we have
\begin{equation}\label{hoho} m\leq p^*_T(|v_m|)\leq K|v_m|^{\alpha+1}.\end{equation}
(For $m=1$, we may have $v_1=\varepsilon$, and thus the latter inequality will not hold.)

Pick $M$ such that
\[M>\max\{K(\alpha+1)2^{\alpha+2}; 2\}.\]
We  now show that there exists a language $S\subset \A^*$ with
$p_S(n)\leq\lceil Mn^\alpha \rceil$ for each $n\geq 1$, and
$T\subseteq S^2$. To prove this we define inductively a nested
sequence of sets $S_1\subseteq S_2\subseteq S_3 \subseteq \cdots$
with $S_m\subseteq \A^*$ such that for each $m\geq 1$ the
following three conditions are satisfied:

i)  $\card(S_m)\leq 2m$,

ii) $p_{S_m}(n)\leq \lceil Mn^\alpha \rceil$ for each $n\geq 1$,

iii) $\{v_1,v_2,\ldots ,v_m\}\subseteq S_m^2$.

\noindent For $m=1$, we consider the factorization
$v_1=\varepsilon\cdot v_1$ and put $S_1=\{\varepsilon, v_1\}$.
Then clearly $S_1$ satisfies each of the conditions i), ii) and
iii) above. For the inductive step, suppose for $m\geq 1$ we have
constructed sets $S_1\subseteq S_2\subseteq\cdots \subseteq S_m$
with the required properties. We say that $n\geq 1$ is a forbidden
length if $p_{S_m}(n)=\lceil Mn^\alpha \rceil$, i.e., in
constructing $S_{m+1}$ from $S_m$ we cannot add to $S_m$ any word
of forbidden length without violating condition ii) at level
$m+1$. Note that 0 is never a forbidden length since there exists only one word of length 0, $\varepsilon$, and nothing else can be added to the set of words of length 0.

Let $F$ denote the set of all forbidden lengths. For
each $0\leq i \leq |v_{m+1}|$ we can factor $v_{m+1}$ as
$v_{m+1}=x_iy_i$, with $|x_i|=i$. We claim that there exists
$0\leq j \leq \left\lceil \frac{|v_{m+1}|}{2}\right\rceil-1$ such that neither
$|x_j|$ nor $|y_j|$ belongs to $F$.So, we can take
$S_{m+1}=S_m\cup \{x_j,y_j\}$. To prove the claim, suppose to the
contrary that for each $0\leq i\leq \left\lceil
\frac{|v_{m+1}|}{2}\right\rceil-1$ there exists $n_i\in \{i,
|v_{m+1}|-i\}\cap F$. Then summing up the number of elements in
$S_m$ of forbidden lengths we obtain:

\begin{eqnarray*}
\card(S_m)&\geq&\sum _{n\in F} \lceil Mn^\alpha \rceil \geq  \sum_{i=0}^{\left\lceil \frac{|v_{m+1}|}{2}\right\rceil-1} \lceil Mn_i^\alpha \rceil 
\geq \sum_{i=0}^{\left\lceil \frac{|v_{m+1}|}{2}\right\rceil-1} Mn_i^\alpha
\geq  \sum_{i=1}^{\left\lceil \frac{|v_{m+1}|}{2}\right\rceil-1} Mi^\alpha +M |v_{m+1}|^{\alpha}.
\end{eqnarray*}
The latter inequality holds since 0 is never a forbidden length, and thus $n_0=|v_{m+1}|$. 
Continuing the chain of inequalities, we see that
\begin{eqnarray*}
\card(S_m)&\geq &\sum_{i=1}^{\left\lceil \frac{|v_{m+1}|}{2}\right\rceil-1} Mi^\alpha +M |v_{m+1}|^{\alpha} >\sum_{i=1}^{\left\lceil \frac{|v_{m+1}|}{2}\right\rceil} Mi^\alpha \geq
\int_{0}^{\frac{|v_{m+1}|}{2}}Mx^\alpha\,dx\\ 
&\geq & \frac{M}{(\alpha+1)}\left(\frac{|v_{m+1}|}{2}\right)^{\alpha +1}
>\frac{K(\alpha+1)2^{\alpha+2}}{(\alpha+1)}\left(\frac{|v_{m+1}|}{2}\right)^{\alpha +1}\\
&\geq&  2K|v_{m+1}|^{\alpha +1} \geq 2K|v_{m}|^{\alpha +1} 
\geq  2m,
\end{eqnarray*}
where the last inequality follows from \eqref{hoho}),
contradicting i). This completes the inductive step. Having
defined the nested sequence $(S_m)_{m\geq 1}$, we set
$S=\bigcup_{m\geq 1}S_m$. Then $p_S(n)=O(n^\alpha)$ and
$T\subseteq S^2$. \end{proof} 

We now return to the proof of Proposition~\ref{t:l1l2} and establish the remaining inequality
$d_\alpha (L)\leq 2d^*_{\alpha+1}(L).$
Let us assume $d^*_{\alpha+1}(L)=k$ for some positive integer $k$. Then
$L\subseteq T^k$ for some $T\in\L^*(\alpha+1)$. By Lemma~\ref{ST}
there exists $S\in \L(\alpha)$ such that $T\subseteq S^2$. Thus
$L\subseteq S^{2k}$ whence $d_\alpha (L)\leq
2k=2d^*_{\alpha+1}(L)$ as required.\end{proof}

The next statement follows immediately from the second double inequality of Proposition \ref{t:l1l2}.

\begin{corollary}
 For any language $L \subseteq \A^*$,
\begin{enumerate}
 \item if $c(L)>0$, then $c^*(L)=c(L)+1$;
\item if $c(L)=0$, then $0 \leq c^*(L)\leq 1$.
\end{enumerate}

\end{corollary}

The next proposition establishes a first relationship between $d_\alpha$   and  complexity:

\begin{proposition}\label{l1prelim}
 Let $\alpha \geq 0$ and $L\subseteq \A^*.$   If $d_\alpha(L)=k$  for some positive integer $k$, then $L\in \L(k(\alpha +1)-1).$  In particular, if $x\in \A^\nats$ and $L=\Fac(x),$ then by taking  $\alpha=0$ we have that if $d_0(x)=k$, then $x\in \W(k-1)$.%\todo{Choose the better of the two alternative proofs!}
\end{proposition}

\begin{proof}  It suffices to prove the proposition for languages $L.$ The result is clear in case $k=1.$ So let us fix $k\geq 2,$ and let $L\subseteq S^k$ for some $S\in \L(\alpha).$  Then there exists a positive integer $C$  such that $p_S(n)\leq Cn^\alpha$ for each $n\geq 0.$  Let $u\in L$ and put $n=|u|.$ Then $u$ is a concatenation of $k$ elements of $S.$ 
We claim  there are ${n+k-1\choose k-1}$ ways of factoring $u=v_1v_2\cdots v_{k}$ with $|v_i|\geq 0.$ In fact, each such factorization of $u$ corresponds to a vector $(n_1,n_2,\ldots ,n_{k})$ with $n_i\geq 0$ and $n_1+n_2+\cdots n_{k}=n.$ The mapping $(n_1,n_2,\ldots ,n_{k})\mapsto (n_1+1,n_2+1,\ldots ,n_{k}+1)$ defines a bijection between the sets $A=\{(n_1,n_2,\ldots ,n_{k})\,|\, n_i\geq 0,\,\,n_1+n_2+\cdots n_{k}=n\}$ and  
$B=\{(m_1,m_2,\ldots ,m_{k})\,|\, m_i\geq 1,\,\,m_1+m_2+\cdots m_{k}=n+k\}.$ Since each element of $B$ corresponds to a partition of $n+k$ consecutive points into $k$ non-empty parts, and since each such partition is given by choosing $k-1$ separation points amongst the $n+k-1$ possible separation points, we deduce that $\card(A)=\card(B)=  {n+k-1\choose k-1}.$
Having established that there are ${n+k-1\choose k-1}=O(n^{k-1})$ ways of factoring $u=v_1v_2\cdots v_{k}$ with $|v_i|\geq 0,$ as each $v_i\in S,$ there are $C|v_i|^\alpha$ choices for each $v_i.$ Thus $p_L(n)=O(n^{k(\alpha +1)-1})$ as required.
\end{proof}

\noindent As an immediate consequence we get:

\begin{corollary}\label{l1} For each language $L\subseteq \A^*$ (resp., infinite word $x\in \A^\nats)$ we have
$c(L)<+\infty$ if and only if $L\in \L(\alpha)$ (resp., $x\in \W(\alpha))$ for some $\alpha \geq 0$.
\end{corollary}

\begin{proof} If $c(L)<+\infty$, then $d_\alpha(L)<+\infty$ for each $\alpha >c(L)$. Fix $\alpha >c(L)$ and a positive integer $k$ such that $d_\alpha(L)=k$. Then by  Proposition~\ref{l1prelim} $L\in \L(k(\alpha +1)-1)$.
The converse follows from Lemma~\ref{cdim}.\end{proof}

\noindent In view of the next corollary, we restrict ourselves henceforth to languages and words of entropy zero.
\begin{corollary}
Languages of positive entropy have cost equal to $+\infty$.
\end{corollary}

Proposition~\ref{l1prelim} suggests that a priori there is no polynomial bound on the complexity of infinite words of cost equal to $0$. The following proposition shows that for each $k\geq 1$ there exists a word $x$ of complexity $\Omega(n^{k-1})$
with $d_0(x)=k$ and hence in particular $c(x)=0$.

\begin{proposition} For each $k\geq 1$ there exists a word $x$ of complexity $\Omega(n^{k-1})$ of cost $0$ and cost dimension $k$.
\end{proposition}

\begin{proof}
For $k=1$ we may simply take the constant word $x=a^\omega,$ and for $k=2$ it suffices to take $x$ to be any Sturmian word (see Example~\ref{e:st}). Thus we may assume that $k \geq 3.$
We construct a word $x$ on the alphabet $\{0,1,\ldots ,k-2\}$ as follows:  We enumerate 
\[\{1,\ldots ,k-2\}^+=\{t_1,t_2,t_3,\ldots\}\]
where the $t_i$ are listed in increasing order, where $\{1,\ldots ,k-2\}^+$ 
is ordered by $t_i<t_j$ if and only if either $|t_i|<|t_j|$ or in case $|t_i|=|t_j|$ then $t_i$ is less than $t_j$ relative to the lexicographic order. So the sequence $t_1,t_2,\ldots$ looks like  $1,2,\ldots,k-2, 11, 12, \ldots$ Then $x\in \{0,1,\ldots ,k-2\}^\nats$ is defined by

\[x=t_0 t_1 t_0 t_2 t_0 t_1 t_0 t_3...,\]
where $t_0=0.$ In other words $x$ is
obtained as the limit of a sequence $(w_n)$ defined by $w_0=t_0$, $w_{n+1}=w_n t_{n+1} w_n$ for all $n \geq 0$.
We claim that the complexity of $x$ is  $\Omega(n^{k-1})$. Indeed, let us restrict ourselves to factors of $x$ of length $n$ which contain a complete factor $0 t_p 0$, where the length of $t_p$ is at least $n/2$. Such a factor of $x$ exists for each $t_p=1^{j_1} 2^{j_2}... (k-2)^{j_{k-2}}$ (that is, for each $j_1,...,j_{k-2}$ under the condition $j_1+\cdots +j_{k-2}=|t_p|\geq n/2$), and for each starting point of that occurrence of $t_p$, which is any number between 1 and $n-|t_p|-1$. So, we have $k-1$ degrees of freedom, and thus the complexity of $x$ is at least $O(n^{k-1})$.
On the other hand, take a factor $w$ of $x$ and find in it a word $t_p$, where $p$ is maximal possible. Here incomplete intersections count: we just fix an occurrence of $w$ to $x$, see what words $t_p$ it intersects and choose the greatest $p$.
If $t_p$ is completely in $w$, it is followed in it by a prefix of $x$. Denote the set of prefixes of $x$ by $S_{k-1}$. Symmetrically, just before $t_p$ in $w$, if it is taken from the beginning, there is a suffix of some word $w_m$ (and $w_m$ are suffixes one of another). We denote the set of these suffixes by $S_0$. As for $t_p$ itself, it belongs to the concatenation of $1^*=S_1$, $2^*=S_2$, etc.; so,
\begin{equation}\label{e:012}
w \in S_0 S_1 \ldots S_{k-2} S_{k-1},
\end{equation}
where the complexity of each $S_i$ is 1. 

If $t_p$ is not completely contained in $w$, three situations are possible. Either $w=t's$, where $t'$ is a suffix of $t_p$; then $t'\in i^*(i+1)^*\cdots (k-2)^*$ for some $i\in\{1,\ldots,k-2\}$, $s$ is a prefix of $x$, and thus $w\in S_i\cdots S_{k-2} S_{k-1} \subset S_0 S_1 \ldots S_{k-2} S_{k-1}$. Or, symmetrically, $w=pt''$, where $t''$ is a prefix of $t_p$; then $t'' \in 1^*2^*\cdots i^*$ for some $i\in\{1,\ldots,k-2\}$, $p$ is a suffix of some $w_m$, and thus $w\in S_0 S_1 \cdots S_{i} \subset S_0 S_1 \ldots S_{k-2} S_{k-1}$. Or, at last, $w$ is a factor of $t_p$, and then $w\in i^*(i+1)^*\cdots j^*$ for some $i,j\in\{1,\ldots,k-2\}$, $i\leq j$, and thus $w \in S_i S_{i+1}\cdots S_j \subset S_0 S_1 \ldots S_{k-2} S_{k-1}$. In all the cases, \eqref{e:012} holds.\end{proof}

While the definition of $x$ in the previous proposition is on a alphabet size which varies with $k,$ by applying to $x$ the morphism 
 $f: i \to 1^{i}0^{k-i}$ we  obtain an infinite binary word satisfying the same required properties.  
 
 We end this section by noting that the set $S$ in Definition~\ref{new} is not assumed to be factorial. 
In fact, as the following proposition shows, this is too strong of a condition:

\begin{proposition}\label{too strong}Let $x\in \A^\nats.$ Suppose $\Fac(x)\subseteq S^k$ for some factorial language $S$ and positive integer $k.$ Then there exists a suffix $y$ of $x$ such that $\Fac(y)\subseteq S.$ In particular, for each positive integer $\alpha \geq 0,$ if $S\in \L(\alpha)$ then $x\in \W(\alpha).$ 
\end{proposition}  

\begin{proof} We remark that if $S$ is factorial, then so is $S^k$ for each $k\geq 1.$ Let $k\geq1$ be the least positive integer such that $\Fac(x)\subseteq S^k.$ The result is clear in case $k=1,$ so we may suppose $k>1.$ By minimality of $k,$ there exists a factor $u$ of $x$ not belonging to $S^{k-1}.$ 
Pick $y\in \A^\nats$ such that $uy$ is a suffix of $x.$ We claim $\Fac(y)\subseteq S.$ Since $S$ is factorial, it suffices to show that every prefix of $y$ belongs to $S.$ So let $z\in \A^*$ be a prefix of $y.$ Then we can write $uz=v_1v_2\cdots v_{k}$ for some $v_i\in S.$ Since $S^{k-1}$ is factorial and $u\notin S^{k-1},$ it follows that $v_1v_2\cdots v_{k-1}$ is a proper prefix of $u$ and hence $z$ is a proper suffix of $v_k.$ Thus $z\in S$ as required.\end{proof}

\section{A characterisation of words of linear complexity in terms of cost dimension}

In this section we characterize words of linear complexity in
terms of the cost dimension. Let $x\in \A^\nats \cup \A^\ints$.  For each $n\geq 0$, let
$\R_x(n)$ denote the set of  right special factors of $x$ of
length $n$ and  $\R_x=\bigcup_{n\geq 0}\R_x(n)$. %In this section, we prove the main result of this paper, namely,
%The $\subseteq$ inclusion has been proven in Proposition \ref{l1}.
%Since for periodic words the statement is obvious, it remains to find the languages $S,T$ of bounded complexity for a given infinite word $u$ of linear complexity $p_u(n)\leq Cn$ such that  \[{\rm Fac(}u{\rm)}\subseteq ST.\]

%In the rest of the section, we consider a non-periodic infinite word $u$ of complexity $P_u(n)\leq Cn$. Our goal is to find appropriate sets $S$ and $T$ and to prove that their complexity is indeed bounded.
%The construction of the sets  $S$ and $T$ is based on so-called {\it markers} which we define below.

%For the sake of the proof of Theorem~\ref{main}, it will be necessary to formulate some of our preliminary results in terms of bi-infinite words $x\in \A^\ints$. As for right infinite words, we say $x\in \A^\ints$ is aperiodic if no suffix of $x$ is purely periodic, i.e., $x$ does not admit a suffix of the form $u^\omega$ for some $u\in \A^+$.

\begin{definition}\label{markers} \rm{Let $D$ be a positive integer. A subset $M\subseteq \A^*$ is called a $D$-{\it marker set} for $x$ if for each $n\geq 1$ and each factor $u$ of $x$ of length $|u|\geq Dn$ we have $\Fac(u)\cap M \cap \A^n\neq \emptyset$. The elements of $M$ are called $D$-markers. }
\end{definition}

%Recall that a factor $v$ of $u$ is called {\it right special} if $va, vb \in$ Fac$(u)$ for at least two different symbols $a,b$.
\begin{lemma}\label{rmarkers} Let $C$ be a positive integer. Then for each aperiodic word  $x\in \A^\nats \cup \A^\ints$ with $p_x(n)\leq Cn$ for each $n\geq 1$,  the set $\R_x$ is a $(C+1)$-marker set for $x$.
\end{lemma}
\begin{proof}Fix a positive integer $n$, and let $u$ be any factor of $x$ of length $(C+1)n$. We show that $u$ contains some element of $\R_x(n)$.  Since $p_x(n)\leq Cn$, and there are $Cn+1$ positions for factors of length $n$ in $u$, by the pigeon-hole principle there exists a factor $v$ of $x$ of length $n$ which occurs in $u$ at least twice. Thus $u$ contains as a factor a word $w$ of length $|w|>n$ which begins and ends in $v$. Hence there exists a prefix $w'$ of $w$ of length $|w'|\geq n$ which is a right special factor of $x$. Otherwise, every occurrence of $v$ in $x$ is an occurrence of $w$, whence $x$ is ultimately periodic, a contradiction.  It follows that the suffix $w''$ of $w'$ of length $n$ belongs to $\R_x(n)$. \end{proof}

%This lemma will be used below together with the following result of the first author from \cite{cas_lin} (see also \cite{cn}):\begin{theorem}\label{cassaigne} Let $C$ be a positive integer. Then for each aperiodic word  $x\in \A^\nats$ with $p_x(n)\leq Cn$ for $n \geq 1$, there exists a constant $K$ (which is a polynomial function in $C)$ such that $\card(\R_x(n))\leq K$ for each $n\geq 0$.\end{theorem}

\noindent The following proposition gives an alternative and more general method for constructing marker sets whose complexity is related to the complexity of the underlying word:

\begin{proposition}\label{l:4n_over_n}
 For each aperiodic  word $x\in \A^\nats\cup \A^\ints$ there exists a $3$-marker set  $M$ for $x$ with \[p_M(n)\leq \frac{p_x(4n)}n\]
 for each $n\geq 1$.
\end{proposition}

\begin{proof} For each $n\geq 1,$ we build recursively (relative to the index $i)$ sets $M_n(i)$ consisting of factors of $x$ of length $n,$ and  $W_n(i)$  consisting factors of $x$ of length $3n.$ In each case $\card(M_n(i))=\card(W_n(i))\leq i.$  The process terminates when each factor of $x$ of length $3n$ contains a factor from $M_n(i)$.  Starting with $M_n(0)$ and $W_n(0)$ both empty, let $w_1$ be the factor of $x$ of length $3n$ beginning in position $n$, and let $m_1$ be the middle block of $w_1$ of length $n$, i.e.,  $w_1=x[n, 4n-1]$  and $m_1=w_1[n,2n-1]=x[2n,3n-1]$.   Then set $W_n(1)=\{w_1\}$ and $M_n(1)=\{m_1\}$.

For the inductive step, fix $i\geq 1$ and suppose we have constructed sets $M_n(i)$ and $W_n(i)$ as required.
Consider the factors of $x$ of length $3n$. If each of them contains a factor from $M_n(i)$, then we are done and we set $M_n=M_n(i)$, $W_n=W_n(i)$.  Otherwise, pick a factor $w_{i+1}$ of $x$ of length $3n$ not containing any element of $M_n(i)$ and set
$W_n(i+1)=W_n(i)\cup \{w_{i+1}\}$ and $M_n(i+1)=M_n(i)\cup \{m_{i+1}\}$ where $m_{i+1}$ is the middle block of $w_{i+1}$ of length $n$.  Note that if $x$ is a one-sided infinite word, then $w_{i+1}=x[m,m+3n-1]$ where $m\geq n$.
Since all $w_i$ are distinct and there are a finite number of factors of $x$ of length $3n$, this process terminates at some point $i\geq 1$. Finally, we set $M=\cup_{n\geq 1}M_n$.  It remains to prove the upper bound on the complexity of $M$.

For each element $w_i$ of $W_n$, we consider a final occurrence $w_i=x[k_i, k_i+3n-1]$ of $w_i$ in $x$. Since $x$ is  aperiodic, each factor of $x$ admits at least one final occurrence in $x$.   Now for each $j=0,\ldots,n-1$ consider its {\it covering factor} $c(i,j)=x[k_i+j-n, k_i+3n+j-1]$. Then the length of $c(i,j)$ is $4n$ and $w_i=c(i,j)[n-j, 4n-j-1]$. Note that even if $x$ is one-sided infinite, each $c(i,j)$ is well defined since each $w_i$ occurs in $x$ at a position $n$ or greater.

%\begin{claim}\label{cc} 
Now let us prove that if $c(i,j)=c(i',j')$, then $i=i'$ and $j=j'$.
%\end{claim}
%\begin{proof}[Proof of Claim~\ref{cc}] 
Indeed, suppose that $c(i,j)=c(i',j')$ but $i'< i$. Then $w_i=c(i,j)[n-j, 4n-j-1]$. %and thus $m_i=c(i,j)[2n-j, 3n-j-1]$.
Analogously, $w_{i'}=c(i,j)[n-j', 4n-j'-1]$ and thus $m_{i'}=c(i,j)[2n-j', 3n-j'-1]$. But since $j,j' \in \{0,\ldots,n-1\}$, we have $2n-j'+1 \geq n-j+1$ and $3n-j'\leq 4n-j$. So, $m_{i'}$ is a factor of $w_i$, a contradiction to our definition of $w_i$. We have proved that $i=i'$.

Next suppose that $j'<j$. Then $w_i=c(i,j)[n-j, 4n-j-1]=c(i,j)[n-j', 4n-j'-1]$. Consider the word $s=c(i,j)[n-j, 4n-j'-1]$. It is $(j-j')$-periodic, and in particular, its prefix $w_i$ is $(j-j')$-periodic. So, $\pi (w_i)\leq j-j'\leq n$. The prefix occurrence of $w_i$ to $s$ overlaps with the suffix occurrence of $w_i$ to $s$ by $3n-(j-j')\geq 2n >\pi(w_i)$ symbols, and thus $s$ is also $\pi(w_i)$-periodic. In particular, $s$ has a virtual square of length $\pi(w_i)$ at the end of the prefix occurrence of $w_i$, that is, at the position $3n$. But $s$ is a factor of $c(i,j)=x[k_i+j-n, k_i+3n+j-1]$, namely, $s=c(i,j)[n-j, 4n-j'-1]=x[k_i,k_i+3n+j-j'-1]$. So, $x$ has an occurrence of $w_i$ (of length $3n$) at position $k_i$, followed by a virtual square of length $\pi(w_i)$ at position $k_i+3n$. It means exactly that this occurrence of $w_i$ is not final, a contradiction.

So, $c(i,j)\neq c(i',j')$ for $i \neq i'$ or $j \neq j'$. Thus, the total number of covering factors $c(i,j)$ is given by
\[\card\left(\{c(i,j)\,|\, 1\leq i\leq \card(W_n), \, j=0,\ldots,n-1\}\right)=n\card(W_n) =n\card(M_n).\] 
On the other hand, each covering factor $c(i,j)$ is a factor of $x$ of length $4n$ whence their number is bounded above by $p_x(4n)$. Thus
\[p_M(n)=\card(M_n)\leq \frac{p_x(4n)}{n}\]
as required.\end{proof}

%\begin{corollary}\label{c:mark} For each word $x\in \W(1)$, there exists a positive integer $D$ and a $D$-marker set $M$ of $x$ of bounded complexity, i.e., with $M\in \L(0)$.  \end{corollary}

 %\begin{proof} If $x$ is ultimately periodic, then $\Fac(x)$ is uniformly bounded, thus we can take $D=1$ and $M=\Fac(x)$. If $x$ is aperiodic, then the result follows immediately from  Lemma~\ref{rmarkers} and Theorem~\ref{cassaigne}.\end{proof}

 We now state and prove the most general result of this section.

\begin{theorem}\label{t:main} Assume  either $y \in\A^\ints$, or  $y\in \A^\nats$ and is recurrent. Let $D$ be a positive integer and assume that $M$ is a $D$-marker set for $y$. Then there exist languages $S,T\subseteq \A^*$ such that $\Fac(y)\subseteq ST$ and for each $n\geq 2D$ we have

\begin{equation}\label{mainestimate}
p_{S}(n),p_{T}(n)\leq \sum_{k\in I_n\cap \nats}p_M(2^k)\left(1+\frac{4p_y(3n)}{2^k}\right)
\end{equation}
where $I_n=(\log_2\left(\frac{n}{2D}\right), \log_2(2n)]$.
 \end{theorem}

\begin{proof}%[Proof of Theorem~\ref{t:main}]
Let us fix %$y\in \A^\ints$ and 
a $D$-marker set $M$ for $y$. For each $k\geq 1$, let $M_k=\{\mathfrak m\in M\,|\, |\mathfrak m|=2^k\}$. The elements of $M_k$ are called markers of order $k$.

%For each factor $v$ of $x$, let $x\big|_v$ denote the set of all occurrences of $v$ in $x$, i.e., $x\big|_v=\{i\geq 0\,|\, v=x[i,i+|v|-1]\}$.
Consider a factor $v$ of $y$ with  $|v|\geq 2 D$. We shall define a rule for decomposing $v$ as a product
$v=s(v)t(v)$. The sets $S$ and $T$ will then be defined as the collection of all $s(v)$ and all $t(v)$ corresponding to all factors $v$ of $y$ of length $|v|\geq 2 D$. %The rule begins by fixing an occurrence of $v$ in $y$, i.e., by fixing $i\in y\big|_v$.
 Let $k\geq 1$ be the largest positive integer such that  $\Fac(v)\cap M_k\neq \emptyset$, and fix $\mathfrak m\in \Fac(v)\cap M_k$. Thus $\mathfrak m$ is a marker word contained in $v$ of length $|\mathfrak m|=2^k$.
 First suppose  some occurrence of $\mathfrak m$ in $v$ is extremal. In this case, we arbitrarily pick one such occurrence, say at position $j$, and cut $v$ precisely in the middle of this extremal occurrence of $\mathfrak m$  so that
 $s(v)=v[0,j+2^{k-1}-1]$ and $t(v)=v[j+2^{k-1}, |v|-1]$. In case all occurrences of $\mathfrak m$ in $v$ are internal, then again arbitrarily pick one such internal occurrence,  say at position $j$, and  cut $v$ precisely in the middle of this internal occurrence of $\mathfrak m$  so that
 $s(v)=v[0,j+2^{k-1}-1]$ and $t(v)=v[j+2^{k-1}, |v|-1]$ (see Fig.~\ref{f1}). Note that our cutting rule gives preference to extremal occurrences of the marker word.
%\begin{center}
\begin{figure}
\centering \includegraphics[width=0.6\textwidth]{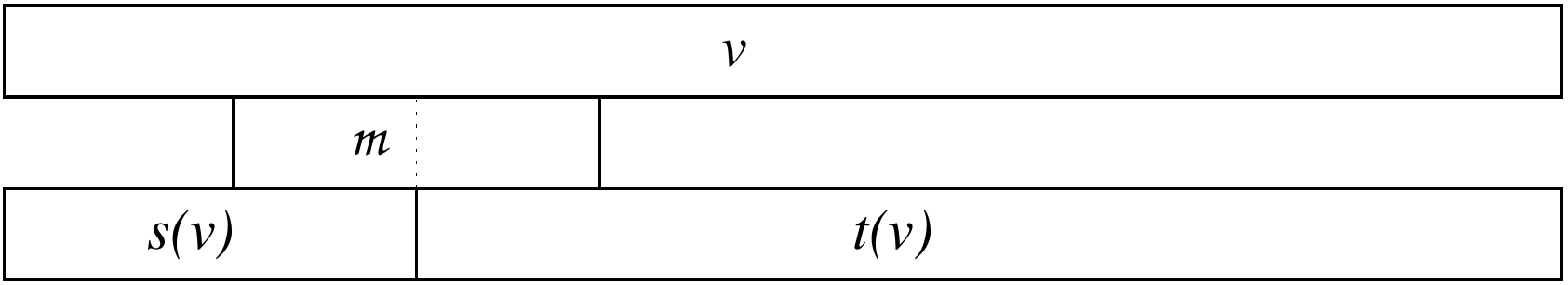}
\caption{Building elements of $S$ and $T$ from a word $v$ and an occurrence of a marker to it}\label{f1}
\end{figure}
%\end{center}
 Now   set
\begin{align*}
S&=(\Fac(y) \cap \A^{< 2D}) \cup \{s(v)\,|\,v \in \Fac(y) \cap \A^{\geq 2D}\}, \\
T&=\{\varepsilon\} \cup \{t(v)\,|\,v \in \Fac(y) \cap \A^{\geq 2D}\},
\end{align*}
where  $\A^{< n}=\bigcup_{k=0}^{n-1} \A^k$ and $\A^{\geq n}=\A^* \backslash \A^{<n}$.

It follows immediately from the definitions that $\Fac(y)\subseteq ST$.  It remains to show that complexities of $S$ and $T$ satisfy \eqref{mainestimate}.  We prove this only for $T$ as the proof for  $S$ works in very much the same way.

Fix $n\geq 2 D$, and let us estimate  $p_T(n)$.  Recall that each $u\in T\cap \A^n$  is obtained by cutting  some factor $v$ of $y$ in the middle of an occurrence of some marker $\mathfrak m$ of maximal order $k$ occurring in $v$
and $u=t(v)$ is the resulting suffix of $v$.  %We begin by showing that the number of possible marker words $\mathfrak m$ is at most $R(\log_2D+3)$. Then we prove that each  marker word contributes at most $2C(4D+2)+1$ elements in each of $ S\cap \A^n$ and  $T\cap \A^n$, and hence the number of elements in each of $ S\cap \A^n$ and  $T\cap \A^n$ is bounded above by $ R(\log_2D+3)(2C(4D+2)+1)$.
Then since $t(v)$ begins with the suffix of $\mathfrak m$ of length $|\mathfrak m|/2$, we have $n\geq 2^{k-1}$.
On the other hand, since $k$ was chosen to be maximal, we have $n<D 2^{k+1}$ for otherwise $v$, which is of length at least $n$, would contain a marker of order $k+1$.
These inequalities combined give
\begin{equation}\label{e:kl}
 \frac{n}{2D}<2^k\leq 2n,
\end{equation}
which implies that $k$ lies in the  interval $I_n=(\log_2\left(\frac{n}{2D}\right), \log_2(2n)]$.  For each such integer $k\in I_n$, the number of marker words of length $2^k$ is equal to $p_M(2^k)$.

%of length $2+\log_2D$ and thus $k$ can take at most $\lfloor \log_2D \rfloor +3\leq \log_2D +3$ integer values. Since there are at most $R$ marker words of each given length, we deduce that there are at most $R(\log_2D+3)$ choices for $\mathfrak m$ which would give rise to an element $u\in S\cap \A^n$ for each $n\geq 2D$. The exact same estimate applies to $T$.

We next prove that each marker word $\mathfrak m$ of length $2^k$ with $k,n$ satisfying \eqref{e:kl} contributes
at most $1+\frac{4p_y(3n)}{2^k}$ elements to $T\cap \A^n$.
%Next we prove that each marker $\mathfrak m$ gives rise to at most $2C(4D+2)+1$ elements in each of $S\cap \A^n$ and $T\cap \A^n$.  We prove this only for $T$ as the proof for  $S$ works in very much the same way. So let us fix $n$ and  a marker word $\mathfrak m$ of order $k$. We may assume that $k$ and $n$ verify \eqref{e:kl} for otherwise $\mathfrak m$ contributes nothing to the cardinality of $T\cap \A^n$.
Let $T(\mathfrak m,n)$ be the set of all $u\in T\cap \A^n$ with $u=t(v)$ for some factor $v$ of $y$ cut at an occurrence of the marker $\mathfrak m$ in $v$.
 We consider separately the three possible types of occurrences of $\mathfrak m:$ internal, initial and final. Thus let
 $T_{\rm int}(\mathfrak m,n)$ (resp., $T_{\rm ini}(\mathfrak m,n)$ and $T_{\rm fin}(\mathfrak m,n))$
 be the subset of $T(\mathfrak m,n)$ arising from internal (resp., initial and final) occurrences of $\mathfrak m$.  Recall that if $t\in  T_{\rm int}(\mathfrak m,n)$, then $t=t(v)$ for some factor $v$ of $y$ in which every occurrence of $\mathfrak m$ in $v$ is internal.
 This implies that $v$ is $\pi(\mathfrak m)$-periodic and hence $t$ is uniquely determined by $\mathfrak m$ and $|t|=n$. More precisely, $t$ is the word of length $n$ occurring at position $2^{k-1}$ of the periodic word $p^\omega$, where $p$ is the prefix of $\mathfrak m$ of length $\pi(\mathfrak m)$ (see Fig.\ref{f:int}).
Thus $\card(T_{\rm int}(\mathfrak m,n))=1$.
\begin{center}
\begin{figure}
\centering \includegraphics[width=0.6\textwidth]{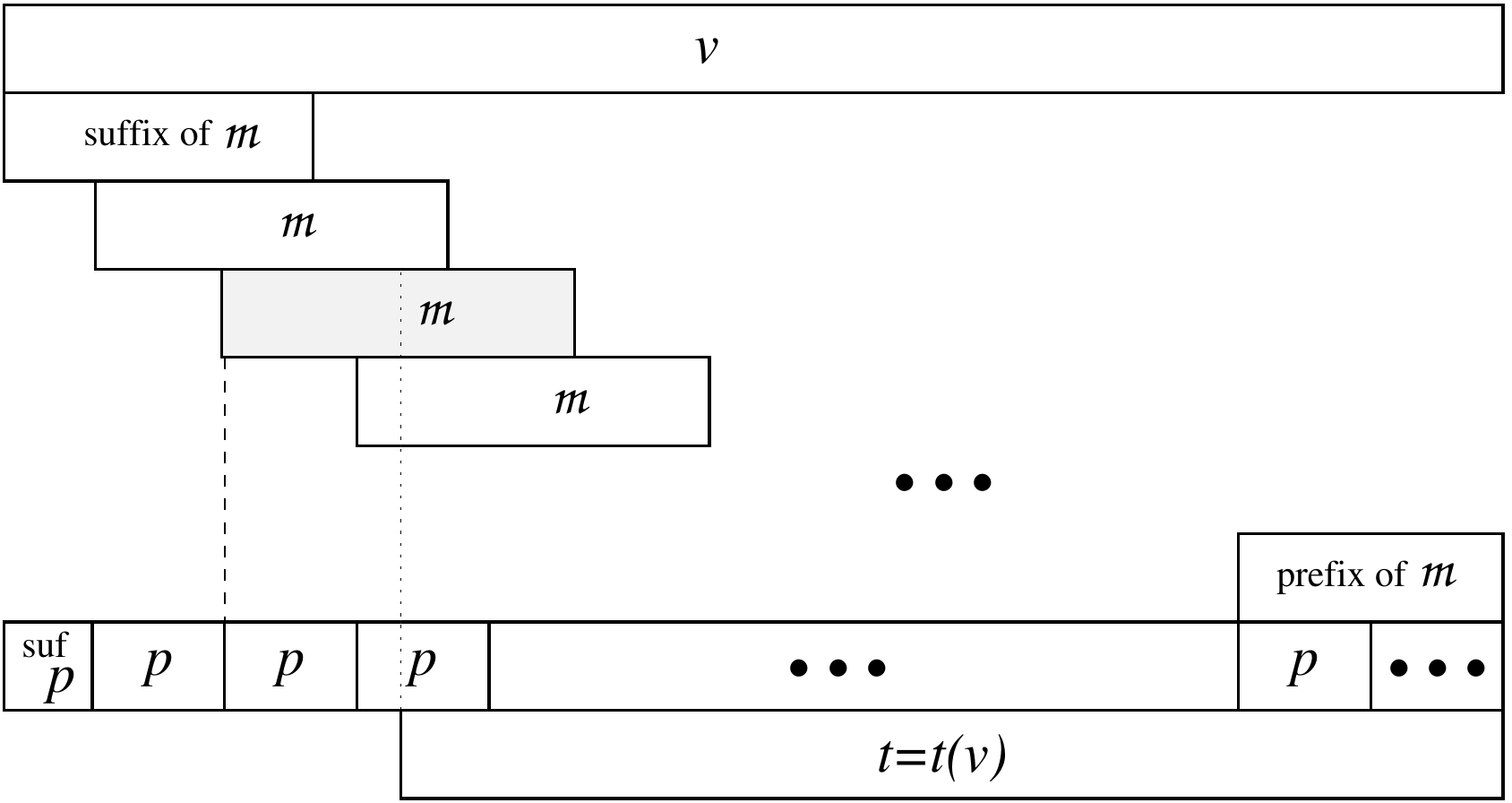}
\caption{The case of an internal occurrence, the unique $t$ is determined by $\mathfrak m$ and the length}\label{f:int}
\end{figure}
\end{center}
\noindent Next we estimate $\card(T_{\rm ini}(\mathfrak m,n))$.

\begin{lemma} For each $n\geq 2D$ we have \[\card(T_{\rm ini}(\mathfrak m,n))\leq \frac{2p_y(3n)}{2^k}\]
\end{lemma}

\begin{proof} %The result is clear in case $T_{\rm ini}(\mathfrak m,n))$ is empty. Thus suppose it is non-empty. Then 
For $t\in T_{\rm ini}(\mathfrak m,n)$, and each $0\leq i< 2^{k-1}$,    let $E_{\rm ini}(\mathfrak m, n, t, i)$ be the collection of all factors $w$ of $y$ of length $n+2^k$ such that $w$ has an initial occurrence of $\mathfrak m$ at position $i$ and an occurrence of  $t$ in position $i+2^{k-1}$ (see Fig.~\ref{f:ini}).
\begin{center}
\begin{figure}
\centering \includegraphics[width=0.6\textwidth]{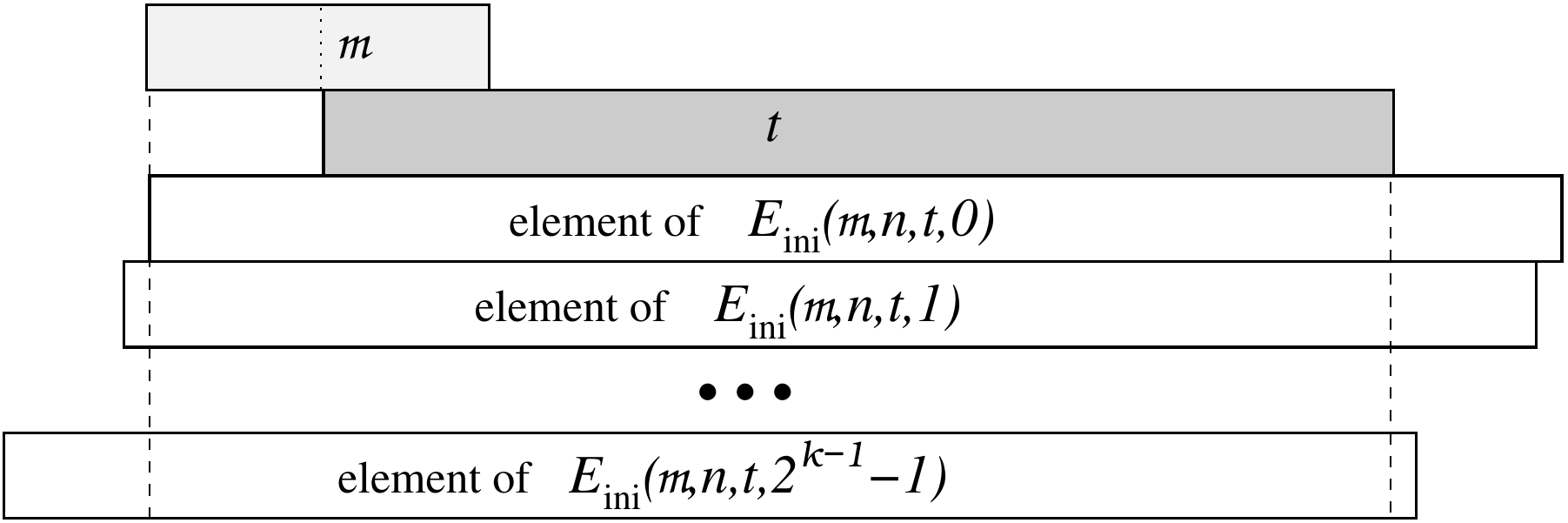}
\caption{The sets $E_{\rm ini}(\mathfrak m, n, t, i)$. The parts between dashed lines are common for all elements}\label{f:ini}
\end{figure}
\end{center}
Let $v$ be a factor of $y$ giving rise to $t$ in $T_{\rm ini}(\mathfrak m,n)$, that is, $v$ contains an initial occurrence of $\mathfrak m$, and the suffix of $v$ starting in the middle of that occurrence of $\mathfrak m$ is $t$.
Since $y$ is assumed either recurrent or bi-infinite, there exists an occurrence of $v$ at the distance more than $i$ from the beginning of the word $y$. So, $E_{\rm ini}(\mathfrak m, n,t,  i)$ is non-empty. Then:

\begin{claim}\label{disjoint}For each $t,t'\in T_{\rm ini}(\mathfrak m,n)$ and $0\leq i,i'<2^{k-1}$, where $t \neq t'$ or $i<i'$, we have \[E_{\rm ini}(\mathfrak m, n,t, i)\cap E_{\rm ini}(\mathfrak m, n,t', i')=\emptyset.\]
\end{claim}

\begin{proof}[Proof of Claim~\ref{disjoint}] Suppose $w\in E_{\rm ini}(\mathfrak m, n,t,  i)\cap E_{\rm ini}(\mathfrak m, n,t', i')$. First consider the case of $0\leq i<i'<2^{k-1}$. Then $\mathfrak m$ occurs in $w$ in position $i$ and $i'$, and since $i'-i<2^{k-1}<|\mathfrak m|$, it follows that the two occurrences of $\mathfrak m$ in $w$ overlap. Since $\mathfrak m$ is $(i'-i)$-periodic, it follows that $\pi(\mathfrak m)\leq i'-i<2^{k-1}<|\mathfrak m|/2$ and hence  $w[i, i'+2^{k}-1]$ is $\pi(\mathfrak m)$-periodic contradicting that the occurrence of $\mathfrak m$ at position $i'$ of $w$ was initial (see Fig.~\ref{f:diff}). So, $i=i'$. But then both $t$ and $t'$ are words of length $n$ occurring in $w$ at position $i+2^{k-1}$, so, $t=t'$. \end{proof}
\begin{center}
\begin{figure}
\centering \includegraphics[width=0.6\textwidth]{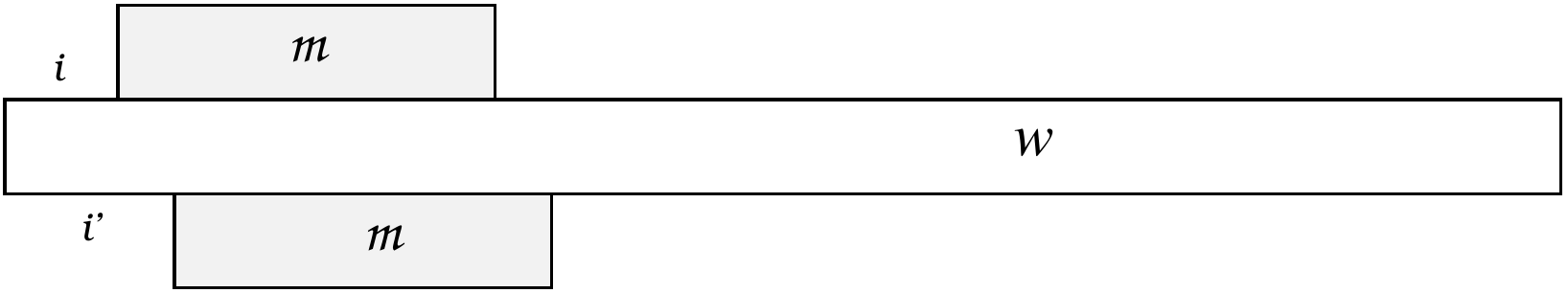}
\caption{Proof of Claim~\ref{disjoint}. The lower occurrence of $\mathfrak m$ is not initial}\label{f:diff}
\end{figure}
\end{center}
So, each $t \in T_{\rm ini}(\mathfrak m,n)$ and each $i\in\{0,\ldots,2^{k-1}-1\}$ correspond to at least one factor of $y$ of length $n+2^k$: the set $E_{\rm ini}(\mathfrak m, n,t, i)$ of all such factors is non-empty, and for different words $t$ or indices $i$, these sets do not intersect. So, 
\[2^{k-1}\card(T_{\rm ini}(\mathfrak m,n))\leq \sum_{i=0}^{2^{k-1}-1}\sum_{t\in T_{\rm ini}(\mathfrak m,n)} \card (E_{\rm ini}(\mathfrak m, n,t, i)) \leq p_y(n+2^k),\]
and since $n+2^k\leq 3n$ and thus $p_y(n+2^k)\leq p_y(3n)$,
%Set
%$ \displaystyle E_{\rm ini}(\mathfrak m, n,t)=\bigcup_{i=0}^{2^{k-1}}E_{\rm ini}(\mathfrak m, n,t,  i)$ \todo{formula moved to take less space}.
%Since the parameter $i$ is allowed to take $2^{k-1}$ distinct values, and since each $E_{\rm ini}(\mathfrak m, n,t,  i)$ is non-empty, it follows from Claim~\ref{disjoint} that \[\card( E_{\rm ini}(\mathfrak m, n,t))\geq 2^{k-1}\card(T_{\rm ini}(\mathfrak m,n)).\]
%On the other hand each element of $E_{\rm ini}(\mathfrak m, n,t)$ is a factor of $y$ of length $n+2^k$ whence
%\[\card( E_{\rm ini}(\mathfrak m, n,t))\leq p_y(n+2^k)\leq p_y(3n).\]
%Combining the two inequalities gives:
%\[2^{k-1}\card(T_{\rm ini}(\mathfrak m,n))\leq p_y(3n)\]
%or equivalently
\[\card(T_{\rm ini}(\mathfrak m,n))\leq \frac{2p_y(3n)}{2^k}\]
as required.\end{proof}

\noindent A similar argument applies to $T_{\rm fin}(\mathfrak m,n))$ and gives the same bound.
Thus in total each $\mathfrak m$ gives rise to at most  $1+\frac{4p_y(3n)}{2^k}$ elements in $T\cap \A^n$ as required.

The arguments for the complexity of $S$ are analogous, completing the proof of Theorem \ref{t:main}.
\end{proof}

\noindent We recall the following result due to the first author from \cite{cas_lin} (see also \cite{cn}): 
\begin{theorem}\label{cassaigne} Let $C$ be a positive integer. Then for each aperiodic word  $x\in \A^\nats$ with $p_x(n)\leq Cn$ for $n \geq 1$, there exists a constant $K$ (which is a polynomial function in $C)$ such that $\card(\R_x(n))\leq K$ for each $n\geq 0$.
\end{theorem}

\noindent We next establish the following classification of words of linear complexity:

\begin{theorem}\label{c:main}Let $x\in \A^\nats$. Then $d_0(x)=2$ if and only if $p_x(n)=\Theta(n)$.  In particular, each $x\in \W(1)$ has cost equal to $0$.
\end{theorem} %\todo{Give examples: f.p primitive morphisms, iets,...}
\begin{proof} One direction follows immediately from Proposition~\ref{l1prelim}. In fact, if $d_0(x)=2$, then applying Proposition ~\ref{l1prelim} with $\alpha=0$ and $k=2$ we deduce that $x\in \W(1)$, i.e., $p_x(n)=O(n)$.
On the other hand, by Lemma~\ref{up} we also have that $x$ is aperiodic, and thus by Morse-Hedlund, $p_x(n)\geq n+1$ for each $n$. Hence, $p_x(n)=\Theta(n)$ as required.

For the converse, suppose $x\in \A^\nats$ and $p_x(n)=\Theta(n)$. Then $x$ is aperiodic for otherwise $p_x(n)=O(1)$.
Since there $x$ is not assumed to be recurrent, to apply Theorem~\ref{t:main} we  will need to replace $x$ by a bi-infinite word. Thus, let $a$ be a  symbol not belonging to $\A$ and define the bi-infinite word $y=\cdots y_{-2}y_{-1}y_0y_1y_2\cdots \in (\A\cup\{a\})^\ints $ by $y_n=x_n$ for $n\geq 0$ and $y_{n}=a$ for each $n\leq -1$. Note that since  $p_y(n)=p_x(n)+n$ and $p_x(n)=\Theta(n)$, it follows that $p_y(n)=\Theta(n)$. Also, since $x$ is aperiodic, then so is $y$. We now apply Theorem~\ref{t:main} to show that there exist languages $S$ and $T$ of bounded complexity such that $\Fac(y)\subseteq ST$.

Fix a positive integer $C$ such that $p_y(n)\leq Cn$ for each $n\geq 1$. Let $M=\R_y$.
By Lemma~\ref{rmarkers}, $M$ is a $D$-marker set for $y$ where $D=C+1$.
By Theorem~\ref{t:main} there exist languages $S$ and $T$ with $p_{S}$,$p_{T}$ satisfying \eqref{mainestimate}
where  $M=\R_y$ and $D=C+1$.

Since $\R_y(n)=\R_x(n) \cup \{a^n\}$ for each $n\geq 0$,
by Theorem~\ref{cassaigne} there exists a positive integer $R$ such that
 $p_M(n)\leq R$ for each $n\geq 0$. Moreover $|I_n|=2+\log_2D$ and thus $k$ takes on at most $3+\log_2D$ possible values.  Furthermore for each such $k$, we have $\frac{1}{2^k}<\frac{2D}{n}$. Thus starting with \eqref{mainestimate}
we have
\begin{eqnarray*}p_{S}(n),p_{T}(n)&\leq & \sum_{k\in I_n \cap \mathbb N}p_M(2^k)\left(1+\frac{4p_y(3n)}{2^k}\right) \\
&\leq &
R(3+\log_2D)\left( 1 + \frac{8Dp_y(3n)}{n}\right)\\
&\leq &R(3+\log_2D)\left( 1 + \frac{24DCn}{n}\right)\\
&=&R(3+\log_2D)\left( 1 + 24DC\right)
\end{eqnarray*}
for each $n \geq 2D$, and hence each of $S$ and $T$ is of bounded complexity. Since $\Fac(x)\subseteq \Fac(y)\subseteq ST$, it follows that $d_0(x)\leq 2$.
But since $x$ is aperiodic, Lemma~\ref{up} implies $d_0(x)\geq 2$.  Hence $d_0(x)=2$ as required.\end{proof}

%{\color{red} S: out of place, should be right after the statement
%of the theorem.}

%\textcolor{green}{
\begin{remark}
The general Theorem \ref{t:main} can also be extended to non-recurrent one-sided infinite words by the same extension argument as Theorem \ref{c:main}.
%This is not done in the paper only for the sake of readability.
\end{remark}
%}

%\textcolor{green}{
\begin{remark}
Since the complexity of a Sturmian word is linear, Theorem \ref{c:main} applies.  However, the general result gives a poorer upper bound on the complexity of $S$ than the one obtained in Example~\ref{e:st}.
\end{remark}
%}

\section{Cost and dimension of words of sub-quadratic complexity}
We begin this section with another corollary of Theorem \ref{t:main} which yields a non-trivial bound on the cost for words of complexity $o(n^2)$ (see Corollary~\ref{c:2a-2}).

\begin{corollary}\label{c:general}Assume either $x\in \A^\ints$ and is aperiodic, or $x\in \A^\nats$ and is both recurrent and aperiodic.
Then there exist languages $S,T\subseteq \A^*$ with
$\Fac(x)\subseteq ST$ and
\[p_{S}(n),p_{T}(n)\leq \frac{12p_x(8n)}{n}+ \frac{192p_x(8n)p_x(3n)}{n^2}\]
for each $n\geq 6$.
\end{corollary}

\begin{proof}Fix $x\in \A^\nats\cup \A^\ints$. Since $x$ is aperiodic, by Proposition~\ref{l:4n_over_n}, there exists a $3$-marker set  $M$ with $p_M(n)\leq \frac{p_x(4n)}n$. By Theorem~\ref{t:main} there exist languages $S,T\subseteq \A^*$ verifying \eqref{mainestimate}
for $n\geq 6$ where $I_n=(\log_2\left(\frac{n}{6}\right), \log_2(2n)]$. Thus for each $n$, there are at most $4$ possible values for $k$ (say $k_0<k_1<k_2<k_3)$ and each verifies $2^{k_i}>2^i\frac{n}{6}$ or equivalently $\frac{1}{2^{k_i}}<2^{-i}\frac{6}{n}$. For each $0\leq i\leq 3$ we bound the term $p_M(2^{k_i})$  by
\[p_M(2^{k_i})\leq \frac{p_x(4\cdot 2^{k_i})}{2^{k_i}}\leq \frac{p_x(8n)}{2^{k_i}}.\] Thus from \eqref{mainestimate} we have
\begin{eqnarray*}p_{S}(n),p_{T}(n)&\leq & \sum_{k\in I_n \cap \mathbb N}p_M(2^k)\left(1+\frac{4p_x(3n)}{2^k}\right) 
\leq \sum_{i=0}^3 p_M(2^{k_i})\left(1+\frac{4p_x(3n)}{2^{k_i}}\right) \\
&\leq & \sum_{i=0}^3 \frac{p_x(8n)}{2^{k_i}} + \sum_{i=0}^3 \frac{4p_x(8n)p_x(3n)}{2^{2k_i}}
\leq  p_x(8n)\sum_{i=0}^3 \frac{1}{2^{k_i}} + 4p_x(8n)p_x(3n)\sum_{i=0}^3 \frac{1}{2^{2k_i}}\\
%\end{eqnarray*}%attention, depends on page formatting
%\begin{eqnarray*}
%\phantom{p_{S,T}(n)}
&\leq&\frac{6p_x(8n)}{n}\sum_{i=0}^3\frac{1}{2^i} + \frac{144p_x(8n)p_x(3n)}{n^2}\sum_{i=0}^3\frac{1}{2^{2i}}\\
&=& \frac{15}{8}\cdot \frac{6p_x(8n)}{n} + \frac{85}{64}\cdot \frac{144p_x(8n)p_x(3n)}{n^2}
\leq \frac{12p_x(8n)}{n}+ \frac{192p_x(8n)p_x(3n)}{n^2}.
\end{eqnarray*}\end{proof}

\noindent As an immediate consequence we have:
\begin{corollary}\label{c:2a-2} Let $\alpha\geq 1$. Then for each  $x\in \W(\alpha)$ we have $c(x)\leq \min\{\alpha, 2\alpha - 2\}$.
\end{corollary}

\begin{proof} The result is clear in case $x$ is ultimately periodic since $c(x)=0$. Thus we may assume $x$ is aperiodic.
Clearly since $p_x(n)=O(n^\alpha)$, it follows that $c(x)\leq
\alpha$. If $x$ is recurrent, then  by Corollary~\ref{c:general}
taking $p_x(n)=O(n^\alpha)$, there exists languages $S,T$ such
that $\Fac(x)\subseteq ST$ and $p_{S}(n),p_{T}(n)=O(n^{2\alpha -2})$.
Thus $c(x)\leq 2\alpha -2$. If $x$ is not recurrent, then as in
the proof of Theorem~\ref{c:main}, we may replace $x$ by an
aperiodic  bi-infinite word $y$ with $p_y(n)=p_x(n)+n$. Since
$\alpha \geq 1$, it follows that  $p_y(n)=O(n^\alpha)$ and so we
may apply Corollary~\ref{c:general} to $y$ to deduce the existence
of languages $S,T$ with $\Fac(x)\subseteq \Fac(y) \subseteq ST$
and with $p_{S}(n),p_{T}(n)=O(n^{2\alpha -2})$. Whence again $c(x)\leq
2\alpha -2$.
\end{proof}

\noindent As another consequence of Corollary~\ref{c:general} we
have: %\todo{Intuitively, substitutions are intruduced in the
%Thue-Morse example. Do we need a more explicit def.?
%{\color{red}S: YES}} \todo{morphism to substitution}

\begin{corollary}Let $x\in \A^\nats$ be a pure morphic word (see \cite{a_sh}). Then except if the complexity of $x$ is in $\Theta(n^2)$, we have  $d_{\alpha}(x)\leq 2$ for each $\alpha>0$ and hence $c(x)=0$.
\end{corollary}

\begin{proof}By a celebrated result of Pansiot in \cite{pansiot}, see also \cite{cn}, if $x$ is a pure morphic word, then $p_x(n)=\Theta(c_n)$ where $c_n\in \{1, n, n\log\log n, n\log n, n^2\}$. Applying Corollary~\ref{c:general} to each choice of $c_n$ except $c_n=n^2$, gives $\Fac(x)\subseteq ST$ where $p_{S}(n),p_{T}(n)=O(n^\alpha)$ for each $\alpha >0$.
Whence $d_\alpha(x)\leq 2$ for each $\alpha>0$ and hence $c(x)=0$.
\end{proof}

\noindent We suspect that $c(x)=0$ even %when $c_n=n^2$
for fixed points of complexity $O(n^2)$ although we are unable to
prove it. 
We saw that $d_0(x)=1$ if and only if a word $x$ is ultimately periodic, while $d_0(x)=2$ if and only if $p_x(n)=\Theta(n)$.
We now show that Theorem~\ref{c:main} does not extend to infinite words of quadratic complexity by exhibiting an infinite word $u$ of complexity $p_u(n)=\Theta(n^2)$ for which $d_0(u)>3$. But for this same word, we will show that $d_0(u)\leq 6$.

\begin{theorem}\label{babaab} Let $u=\prod_{i=1}^{\infty} ab^i=ababbabbb\cdots$.  Then $p_u(n)=\Theta(n^2)$ and $4\leq d_0(u)\leq 6.$
\end{theorem}

\begin{proof} We begin by observing that the factor complexity of $u$ is  quadratic: $u$ is the second shift of the fixed point beginning in $c$ of the (non-primitive) morphism $a \mapsto ab, b \mapsto b, c \mapsto ca$, considered by Pansiot in \cite{pansiot} (see Theorem~4.1 and Example~1 therein). To show that
$d_0(u)>3$, we actually prove something stronger:

\begin{lemma}\label{w3}  $d^*_1(u)>3$.
\end{lemma}

\begin{proof}
Suppose to the contrary that $d^*_1(u)\leq 3$. Then there exist languages $X,Y,Z\subseteq \{a,b\}^*$
with $p^*_X(n),p^*_Y(n),p^*_Z(n)= O(n)$ and such that $\Fac(u)\subseteq XYZ$.
Thus each factor $v$ of $u$ admits a factorization $v=x(v)y(v)z(v)$ with $x(v)\in X,\,y(v)\in Y$ and $z(v)\in Z$.

For each $k,l\geq 1$ set $w_{k,l}=ab^la b^{l+1} \cdots a b^{l+k-1} a$. Then each $w_{k,l}$ is a factor of $u$ of length
\begin{equation}\label{lengthwkl}
|w_{k,l}|=k\left(l+\frac{k+1}{2}\right)+1.
\end{equation}

\begin{claim}\label{nlog} Let\[E(n)=\{(k,l)\,|\, |w_{k,l}|\leq n,\, k\geq 3,\, l \geq \sqrt{n}\}.\]
Then $\card( E (n)) = \Theta(n \log n)$.
\end{claim}

\begin{proof}[Proof of Claim~\ref{nlog}] Using \eqref{lengthwkl}, we see that the condition $|w_{k,l}|\leq n$ is equivalent to
\[l\leq \frac{n-1}{k}-\frac{k+1}{2}.\]
Thus,
\[\card(E(n))= \sum_{k=3}^{\infty} \card\left( \left \{ l \in \mathbb N: \sqrt{n} \leq l \leq \frac{n-1}{k}-\frac{k+1}{2}\right \}\right). \]
All but finite number of terms of this sum are null. In particular, they are null for $k\geq \sqrt{n}$: in that case,
\[ \frac{n-1}{k}-\frac{k+1}{2} \leq \frac{n}{\sqrt{n}}-\frac{\sqrt{n}+1}{2}<\sqrt{n}.\]
A term number $k$ of the sum is bounded from above by $\frac{n-1}{k}$ and from below by $\frac{n-1}{k}-\frac{k+1}{2}-\sqrt{n}-1$ (this expression can be negative, so the $k$th term is not always equal to it). So, 

\[\card( E (n))\leq \sum_{k=3}^{\lfloor \sqrt{n} \rfloor}\frac{n-1}{k}=\Theta\left ( n \log n\right )\mbox{~and}\]

\[\card( E (n))\geq \sum_{k=3}^{\lfloor \sqrt{n} \rfloor}
\left (\frac{n-1}{k}-\frac{k+1}{2}-\sqrt{n}-1\right )=\Theta\left ( n \log n\right ).\]
%But
% \[\sum_{k=3}^{\lfloor \sqrt{2n} \rfloor} \frac{n-1}{k} =
%(n-1) \sum_{k=3}^{\lfloor \sqrt{2n} \rfloor} \frac{1}{k} = \Theta\left ( n \log n\right ) 
%\]
%while
% \[\sum_{k=3}^{\lfloor \sqrt{2n} \rfloor} \left (\frac{k+1}{2}+\sqrt{n}-1 \right )  =
%\Theta(n).\]
%Claim~\ref{nlog} now follows. 
\end{proof}

\medskip
We say that a factor $v$ of $u$ is {\it of type $(k,l)$} if $v=b^i w_{k,l} b^j$ for some $i,j\geq 0$. Clearly, each factor $v$ of $u$ is either of type $(k,l)$ or contains at most one occurrence of the symbol $a$.

\begin{claim}\label{onlyn} Denote by $F(n)$ the subset of $E(n)$ of pairs $(k,l)$ for which there exists a factor $v$ of $u$ of type $(k,l)$ with $|v|\leq n$  whose decomposition $v=x(v)y(v)z(v)$ satisfies $|x(v)|_a \leq 1$ and $|z(v)|_a \leq 1$. Set $H(n)=E(n)\setminus F(n)$.
Then $\card(H(n))=\Theta(n\log n)$.
\end{claim}

\begin{proof}[Proof of Claim~\ref{onlyn}] Consider the mapping $\varphi_n: F(n)\rightarrow Y$ defined as follows:
For each $(k,l)\in F(n)$,  there exists a factor  $v$ of $u$ of type $(k,l)$ with $|v|\leq n$, $|x(v)|_a \leq 1$ and $|z(v)|_a \leq 1$. Set $\varphi_n((k,l))=y(v)\in Y$.
 Since  $|v|_a=|w_{k,l}|_a=k+1\geq 4$, we have that  $|y(v)|_a\geq k-1\geq 2$. It follows therefore that $y(v)$ is either of type  $(k,l)$, or of type
$(k-1,l+1)$, or of type  $(k-1,l)$, or of type $(k-2,l+1)$.
This implies that for each $y\in Y$ in the image of $\varphi_n$, there are at most four pairs $(k,l)\in F(n)$ which map to $y$.  But by assumption the total number of words in $Y$ of length at most $n$ is $p^*_Y(n)=O(n)$. Thus $\card( F(n))\leq 4p^*_Y(n)=O(n)$. On the other hand by Claim~\ref{nlog}, we have $\card(E(n))=\Theta(n\log n)$. Thus $\card(H(n))=\Theta(n\log n)$. \end{proof}

The next claim gives the asymptotic growth of the number of such factors $v$ of $u$ of type $(k,l)\in H(n)$.

\begin{claim}\label{sn} Let $s(n)$ denote the number of distinct factors $v$ of $u$ of length $|v|\leq n$ whose type belongs to $H(n)$. Then $s(n)=\Omega(n^2\log n)$.
\end{claim}

\begin{proof}[Proof of Claim~\ref{sn}] In view of Claim~\ref{onlyn}, it suffices to show that for each type $(k,l)\in H(n)$ there are at least $n$ factors $v$ of $u$ of length $|v|\leq n$ and of type $(k,l)$. So fix a type $(k,l)\in H(n)$.
Then $v$ is of type $(k,l)$ if and only if $v=b^iw_{k,l}b^j=b^iab^la b^{l+1} \cdots a b^{l+k-1} ab^j$ where
$0\leq i\leq l-1$ and $0\leq j\leq l+k$. Thus there are at least $l$ choices for each of $i$ and $j$. But since $l\geq \sqrt{n}$, we have at least $n$ choices for such $v$.
\end{proof}

Let $v$ be a factor of $u$ of length $|v|\leq n$ whose type belongs to $H(n)$. Then
by definition of $H(n)$,  writing $v=x(v)y(v)z(v)$ we have either $|x(v)|_a\geq 2$ or $|z(v)|_a\geq 2$. In case $|x(v)|_a\geq 2$, then $v$ is uniquely determined by its length and  $x(v)$. Thus the number of such words is bounded above by $np^*_X(n)=O(n^2)$. Similarly, if $|z(v)|_a\geq 2$, then $v$ is uniquely determined by its length and  $z(v)$, and hence the number of such words is also bounded above by $np^*_Z(n)=O(n^2)$. Thus $s(n)=O(n^2)$ in contradiction with Claim~\ref{sn}.
This completes our proof of Lemma~\ref{w3}.\end{proof}

Having established that $d^*_1(u)>3$ it follows from  Proposition~\ref{t:l1l2}   that $d_0(u)>3$ as required. 

\noindent We next show that $d_0(u)\leq 6$.

\begin{proposition}\label{w6} Let $u=\prod_{i=1}^{\infty} ab^i$. Then there exist languages $S_1,S_2,S_3$ and $S_4$ with $S_1, S_4 \in \L(0)$ and $S_2,S_3 \in \L^*(1)$ such that $\Fac(u)\subseteq S_1S_2S_3S_4$.
\end{proposition}

\noindent Combined with Lemma~\ref{ST} and Lemma~\ref{w3}, Proposition \ref{w6} yields:

\begin{corollary} $d^*_1(u)=4$ and $d_0(u)\leq 6$. \end{corollary}

\begin{proof}[Proof of Proposition~\ref{w6}] Given a positive integer $n$, let $\nu_2(n)$ denote the $2$-adic valuation of $n$ defined as  the largest exponent $r$ such that $2^r$ divides $n$. Given positive integers $k\leq l$, there exists a unique
$k\leq j\leq l$ such that $\nu_2(j)\geq \nu_2(i)$ for each $k\leq i\leq l$.

Every factor $v$ of $u$ containing at least two occurrences of the letter $a$ is necessarily of the form
$b^{i} a b^{k} a b^{k+1}
a \cdots  b^{l} a b^{i'}=b^i w_{l-k+1,k} b^{i'}$
for some $1\leq k\leq l$, $0\leq i\leq k-1$ and $0\leq i'\leq l+1$.
Given such a $v$ we factor it as follows:

\[\underbrace{b^{i}} \underbrace{ a b^{k} a b^{k+1}
a \dots a b^{j-1}a}\underbrace{b^j a \dots  a b^{l} a}
\underbrace{ b^{i'}}=\underbrace{b^{i}}\underbrace{w_{j-k,k}}\underbrace{b^j w_{l-j,j+1}}\underbrace{b^{i'}}\]
where $j$ is the unique number between $k$ and
$l$ of maximal $2$-adic valuation. Here by convention $w_{0,k}=a$ for all $k$. Writing $j=2^r(2m+1)$, where $r=\nu_2(j)\geq 0$ and $m \geq 0$, we have  $k>j-2^r=2^{r+1}m$ and $l<j+2^r=2^{r+1}(m+1)$. Thus
\[\Fac(u)\subseteq S_1S_2S_3S_4,\]
where $S_1=S_4=\{b^n\,|\,n\geq 0\}$, and
{\small
\begin{eqnarray*}
S_2&=&\{\varepsilon,a\}\cup \{ab^k a \cdots a b^{2^r(2m+1)-1}a\,|\,r\geq 0,\, m \geq 0,\, 2^r(2m+1)-1\geq k>2^{r+1}m\},\\
S_3&=&\{\varepsilon\}\cup \{b^{2^r(2m+1)}a\cdots ab^l a\,|\,r\geq 0,\, m\geq 0,\,2^r(2m+1)\leq l<2^{r+1}(m+1)  \}.
\end{eqnarray*}
}
Note that by adding $\varepsilon$ to both $S_2$ and $S_3$ allows us to also decompose  factors of $u$ containing fewer than two occurrences of the letter $a$. So for instance, $b^iab^{i'}$ factors as $b^iab^{i'}=b^i\cdot a\cdot \varepsilon \cdot b^{i'}$ and $b^i$ as $b^i=b^i\cdot \varepsilon\cdot \varepsilon\cdot \varepsilon$.
Also note that  $ab^k a \cdots a b^{j-1}a\in S_2$ if and only if $\nu_2(j)= \max \{\nu_2(i)\,|\, k\leq i\leq j\}$, and similarly $b^{j}a\cdots ab^l a\in S_3$ if and only if $\nu_2(j)=\max \{\nu_2(i)\,|\, j\leq i\leq l\}$.

Clearly $p_{S_1}(n)=p_{S_4}(n)=1$ for each $n\geq 0$, whence $S_1,S_4\in \L(0)$. Thus it remains to show that $S_2$ and $S_3$ are each in $\L^*(1)$, i.e., each has linear accumulative complexity.

\begin{claim}\label{bds} Let $s$ be a positive integer. Then for each fixed $r\geq0$ and $m\geq 0$,
\[\card(\{v=ab^k a \cdots a b^{2^r(2m+1)-1}a\in S_2\,|\,2\leq|v|\leq 2^s+1\})\leq \min\left\{ 2^r, \frac{2^s+1}{2^{r+1}m+1}\right\},\] 
\[\card(\{v=b^{2^r(2m+1)}a\cdots ab^l a\in S_3\,|\,2\leq |v|\leq 2^s+1\})\leq \min\left\{ 2^r, \frac{2^s+1}{2^{r}(2m+1)+1}\right\}.\]
 \end{claim}

\begin{proof}[Proof of Claim~\ref{bds}] From the definition of $S_2$, if $v=ab^k a \cdots a b^{2^r(2m+1)-1}a\in S_2\backslash \{\varepsilon,a\}$, then $k$ ranges between $2^{r+1}m+1$ and $2^r(2m+1)-1$.  Thus the number of such $v$ is bounded above by $2^r(2m+1)-1-(2^{r+1}m+1)+1=2^r-1<2^r$. Similarly, if $v=b^{2^r(2m+1)}a\cdots ab^l a\in S_3\backslash \{\varepsilon,a\}$,  then $l$ ranges between $2^r(2m+1)$ and $2^{r+1}(m+1)-1$, thus the number of such $v$ is bounded above by
$2^{r+1}(m+1)-1-2^r(2m+1)+1=2^{r+1}-2^r=2^r$. The second estimate in each case takes into account the restriction on $|v|$ and is obtained by replacing the elements in each set by their lengths. In the case of $S_2$, we are estimating the cardinality of a set of natural numbers whose biggest element is at most $2^s+1$, smallest element is $2^r(2m+1)+1$, and the smallest difference between two elements is $2^{r+1}m+2$ (corresponding to the smallest allowable value of $k)$. Thus the cardinality of the set is bounded above by $\frac{(2^s+1)-(2^r(2m+1)+1)}{2^{r+1}m+2}+1<\frac{2^s+1}{2^{r+1}m+1}$. A similar argument yields the second estimate in the case of $S_3$. \end{proof}

\begin{claim}\label{S2} Let $s$ be a positive integer. Then $p^*_{S_2}(2^s+1)\leq 2 + 2^s(3 + \sqrt 2)$.
\end{claim}

\begin{proof}[Proof of Claim~\ref{S2}] Let $s$ be a positive integer. Let $v\in S_2$ with $|v|\leq 2^s+1$. Then either
$v=\varepsilon$ or $v=a$, or $v=ab^k a \cdots a b^{2^r(2m+1)-1}a$ in which case in particular  $2^r(2m+1)+1 \leq 2^s+1$.
This implies that $0\leq r\leq s$ and $m< 2^{s-r-1}$. Thus either $0\leq r<s$ and $m<2^{s-r-1}$, or $s=r$ and $m=0$.
 In the latter case, $v=ab^{2^s-1}a$ and hence this case contributes just one element to $p^*_{S_2}(2^s+1)$. Thus, adding $v=\varepsilon$ and $v=a$, we obtain the estimate
 \[p^*_{S_2}(2^s+1)\leq 3 +\card(\{v=ab^k a \cdots a b^{2^r(2m+1)-1}a\,|\, |v|\leq 2^s+1, 0\leq r<s \,\,\mbox{and}\,\,\, m<2^{s-r-1}\}).\]

\noindent Applying Claim~\ref{bds} for the number of words $v\in S_2$ of the form $v=ab^k a \cdots a b^{2^r(2m+1)-1}a$ for each parameter value $(r,m)$ yields
\begin{equation}\label{e:s2}
 p^*_{S_2}(2^s+1)\leq 3+\sum_{r=0}^{s-1}\sum_{m=0}^{2^{s-r-1}} \min\left\{ 2^r, \frac{2^s+1}{2^{r+1}m+1}\right\}.
\end{equation}

We extract for each value of $r$ the term corresponding to $m=0$. Since $\min\{2^r,2^s+1\}=2^r$, the contribution to $p^*_{S_2}(2^s+1)$ of all pairs $(r,0)$ is bounded by $\sum_{r=0}^{s-1}2^r=2^s-1$. Hence

\[ p^*_{S_2}(2^s+1)\leq 2+2^s+\sum_{r=0}^{s-1}\sum_{m=1}^{2^{s-r-1}} \min\left\{ 2^r, \frac{2^s+1}{2^{r+1}m+1}\right\}\]
Since $m<2^{s-r-1}$, we have $2^{r+1}m<2^s$ and hence $\frac{2^s+1}{2^{r+1}m+1}<\frac{2^s}{2^{r+1}m}$. Moreover since
for all positive $x,y$ we have $\min(x,y)\leq \sqrt{xy}$, we obtain

\begin{eqnarray*}
p^*_{S_2}(2^s+1)&\leq& 2+2^s+\sum_{r=0}^{s-1}\sum_{m=1}^{2^{s-r-1}} \min\left\{ 2^r, \frac{2^s}{2^{r+1}m}\right\}\\
&\leq&
2+2^s+\sum_{r=0}^{s-1}\sum_{m=1}^{2^{s-r-1}} 2^{\frac{s-1}{2}}\frac{1}{\sqrt{m}}\\
&=&2+2^s+2^{\frac{s-1}{2}}\sum_{r=0}^{s-1}\sum_{m=1}^{2^{s-r-1}}\frac{1}{\sqrt{m}}.
\end{eqnarray*}
Since
\[\sum_{m=1}^{2^{s-r-1}}\frac{1}{\sqrt{m}}\leq \int_{0}^{2^{s-r-1}}\frac{dx}{\sqrt{x}}=2\sqrt{2^{s-r-1}}=2^{\frac{s-r+1}{2}}\]
we obtain
\[p^*_{S_2}(2^s+1)\leq 2+2^s\left(1+\sum_{r=0}^{s-1}2^{-r/2}\right)\leq 2+2^s\left(1+\sum_{r=0}^{\infty}\left(\frac{1}{\sqrt{2}}\right)^r\right)=2+2^s(3+\sqrt{2})\]
as required.\end{proof}

\begin{claim}\label{forn} For each positive integer $n$ we have $p^*_{S_2}(n)\leq 2+n(6+2\sqrt{2})$.
\end{claim}

 \begin{proof}[Proof of Claim~\ref{forn}] For $n<1$, the bound is obvious. Fix a positive integer $n\geq 2$ and pick $s\geq 1$ such that $2^{s-1}<n\leq 2^s$, so that $2^s<2n$.
Using  Claim~\ref{S2} together with the fact that  $p^*_{S_2}$ is a non-decreasing function, we obtain
\[p^*_{S_2}(n)\leq p^*_{S_2}(2^s+1)\leq 2+2^s(3+\sqrt{2})\leq 2+2n(3+\sqrt{2})\leq 2+n(6+2\sqrt{2})\]
as required.  \end{proof}

\noindent It remains to find a linear bound for $p^*_{S_3}(n)$.

\begin{claim}\label{S3} Let $s$ be a positive integer. Then $p^*_{S_3}(2^s+1)\leq 2 + 2^s(3 + \sqrt 2)$. And hence as in Claim~\ref{forn} we have $p^*_{S_3}(n)\leq 2+n(6+2\sqrt{2})$.
\end{claim}

\begin{proof}[Proof of Claim~\ref{S3}] The proof for $S_3$ is analogous to that of $S_2$. Fix a positive integer $s$.
Let $v\in S_3$ with $|v|\leq 2^s+1$. Then either
$v=\varepsilon$ or $v=b^{2^r(2m+1)}a\cdots ab^l a$ in which case  $2^r(2m+1)+1\leq2^s+1$.
As before this implies either $0\leq r<s$ and $m<2^{s-r-1}$, or $s=r$ and $m=0$.
 In the latter case, $v=b^{2^s}a$ and hence this case contributes just one element to $p^*_{S_2}(2^s+1)$. Thus, combined with $v=\varepsilon$, we obtain the estimate
 \[p^*_{S_3}(2^s+1)\leq 2 +\card(\{v=b^{2^r(2m+1)}a\cdots ab^l a\,|\, |v|\leq 2^s+1, 0\leq r<s \,\,\mbox{and}\,\,\, m<2^{s-r-1}\}).\]
Applying Claim~\ref{bds} for the number of words corresponding to each parameter value $(r,m)$ gives
\begin{equation}\label{newbound}p^*_{S_3}(2^s+1)\leq 2+\sum_{r=0}^{s-1}\sum_{m=0}^{2^{s-r-1}} \min\left\{ 2^r, \frac{2^s+1}{2^r(2m+1)+1}\right\}.\end{equation}
The claim now follows by observing that the righthand side of \eqref{newbound} is less than the righthand side of \eqref{e:s2}.
\end{proof}
\noindent Claim~\ref{S3} completes the proof of Proposition~\ref{w6}.
\end{proof}

\noindent This concludes our proof of Theorem~\ref{babaab}.\end{proof}

\section{Positive cost for greater than quadratic complexity }
\iffalse

We begin by posing the following related open questions:

\begin{question}\label{open1} Is $d_0(x)<+\infty$ for each infinite word $x\in \W(2)$ ?\end{question}

\begin{question}\label{open2} Is  each infinite word $x\in \W(2)$  of cost equal to $0$ ?
\end{question}

\noindent An affirmative answer to Question~\ref{open1} implies an affirmative answer to Question~\ref{open2}. While we do not know the answer to either question, the following results suggest that if either admits an affirmative answer, then $\W(2)$  is in some sense optimal:

%In the previous section we exhibited the existence of a word in $\W(2)$ which does not belong to $\W_3$, but as it turns out, belongs to $\W_6$. We do not know of any example of a word in $\W(2)$ which does not belong to some $\W_k$. So we ask the following open question:

\fi
At the moment, we do not know if the cost of a word of quadratic complexity can be greater than 0. However, the next theorem states that for any growth of complexity function which is faster than $Cn^2$, this is possible.
\begin{theorem}\label{n2fn}
Let $f(n)$ be any non-decreasing integer function satisfying $f(1)= 1$, $f(n)\leq n$ and $\lim_{n\rightarrow \infty}f(n)= +\infty$. Then there exists an infinite word $x\in \{a,b\}^\nats$ of complexity $O(n^2 f(n))$ such that if $\Fac(x)\subseteq S^k$ for some $S\subseteq \{a,b\}^*$ and $1\leq k<+\infty$, then 

\[p^*_S(n)=\Omega(\sum_{p=1}^\frac{n-2}{2(2k-1)}f(p)).\]

\end{theorem}

\begin{proof}  Fix a function $g:\nats \times \nats \rightarrow \nats$  satisfying  $g(1,1) \geq 1$, $g(p,q)\leq g(p,q+1)$,  $g(p,f(p))\leq g(p+1,1)$ for all $p,q \in \nats$ and $\lim_{p\rightarrow \infty}g(p,1)=+\infty$.  For instance, we can take $g(p,q)=p^{f(p)}+q$.

Define $x\in \{a,b\}^\nats$ as follows:

\[x=\prod_{p=1}^{\infty} \prod_{q=1}^{f(p)}(a^{p}b^{q})^{g(p,q)}.\]

Fix $k\geq 1$, and suppose  $\Fac(x)\subseteq S^k$ for some language $S\subseteq \{a,b\}^*$.
%Fix a positive integer $M$ such that $p_S(n)\leq M$ for all $n\geq 0$. Whence $p^*_S(n)\leq nM$ for all $n\geq 0$.

\begin{claim}\label{c:pq}
For every triple of positive integers $n, p,q$ verifying $(p+q)(2k-1)\leq n-2$, $q\leq f(p)$ and $g(p,q)\geq 2k-1$, the set $S$ contains a factor $s_{p,q}$ of  $b (a^p b^q)^{2k-1} a$ of length $|s_{p,q}|\leq n$ containing $ba^pb^qa$ as a factor.
Moreover, $s_{p,q}\neq s_{p',q'}$ whenever $(p,q)\neq (p',q')$.
\end{claim}

\begin{proof}[Proof of Claim~\ref{c:pq}]Since $g(p,q)\geq 2k-1$ and $q\leq f(p)$, the word $b (a^p b^q)^{2k-1} a$ is a factor of $x$.
Moreover since $(p+q)(2k-1)\leq n-2$, we have that $|b (a^p b^q)^{2k-1} a|\leq n$.
Given any factorization $b (a^p b^q)^{2k-1} a=u_1u_2\cdots u_k$ with $u_i\in \{a,b\}^*$, we see that of $2k$ occurrences of $ba$, at most $k-1$ lie accross boundaries of $u_i$. It remains $k+1$ occurrences of $ba$, and so two of them lie in the same $u_j$. This means that  $u_j$ contains $ba^p b^q a$ as a factor and we can take $s_{p,q}=u_j$.  \end{proof}

Let \[P(n)=\{(p,q)\,|\, (p+q)(2k-1)\leq n-2,\,q\leq f(p),\, g(p,q)\geq 2k-1\}.\] By Claim~\ref{c:pq},
there exists an injection
\[P(n) \hookrightarrow \{v\in S\,|\, |v|\leq n\}\]
given by $(p,q)\mapsto s_{p,q}$. We now estimate, for each $n$ sufficiently large,  the cardinality of the set $P(n)$.
Since the function $g(p,q)$ is non-decreasing on $p$ and $q$, and $g(p,1)\to +\infty$, there exists a positive integer $p_0$ such that $g(p,q)\geq 2k-1$ for all $p \geq p_0$ and all $q$. Since $f(p)\leq p$ for all $p$, for any $q\leq f(p)$ we have $p+q \leq p+f(p) \leq 2p$. In other words, any $p$ between $p_0$ and $\frac{n-2}{2(2k-1)}$ satisfies the conditions $(p+q)(2k-1)\leq n-2$ and  $g(p,q)\geq 2k-1$. Since for each such $p$ there are $f(p)$ possible values for the second coordinate $q$,
for all $n$ sufficiently large we have
\[p^*_S(n)\geq \card(P(n))\geq  \sum_{p=p_0}^\frac{n-2}{2(2k-1)} f(p).\]
Whence
\[p^*_S(n)=\Omega(\sum_{p=0}^\frac{n-2}{2(2k-1)}f(p)).\]

 %Since $f(p) \to \infty$ with $p$, we have
%\begin{eqnarray}\label{fp}p^*_S(n)\geq \sum_{p=p_0}^\frac{n-2}{2(2k-1)} f(p) > Mn\end{eqnarray}
%for some $n$ sufficiently large, a contradiction to the fact that $p^*_S(n)\leq Mn$ for all $n$.
%Thus this proves that $x\notin \W_k$.

It remains to show that  the factor complexity of $x$ is $O(n^2 f(n))$. For this purpose we partition the factors of $x$ into four groups and estimate the number of factors of length $n$ in each group. Each factor $v$ of $x$ belongs to one or more of the following groups:

\begin{itemize}

\item group 1: factors of a block of the form $(a^{p}b^{q})^j$ for some $p$, $q$ and $j$.

\item group 2: factors of a block of the form $(a^{p}b^{q})^{k_1}(a^{p}b^{q+1})^{k_{2}}$.

\item group 3: factors of a block of the form $(a^{p}b^{f(p)})^{k_1}(a^{p+1}b)^{k_{2}}$.

\item group 4: factors containing some complete block  $(a^{p}b^{q})^{g(p,q)}$ as a factor.
\end{itemize}
We note that some of these groups overlap, which is not a
problem since we seek only an upper bound on the factor complexity. We estimate the number of words of length $n$ in each group.

In group 1, we have $O(n)$ words of the form $a^ib^{n-i}$ or $b^ia^{n-i}$, plus $O(n^2)$ words of the form $a^i b^q a^{n-q-i}$ (uniquely determined by $i\geq 1,q<n$) or $b^i a^p b^{n-p-i}$ (uniquely determined by $i\geq 1,p<n$), plus words containing factors of the form $b a^p b^q a$ or $a b^q a^p b$. These last set of words are uniquely determined by $p<n$, $q \leq f(p)$ and the position of the first occurrence of $a^p$, which takes values between 0 and $p+q-1<n$. Thus, the number of such words (and thus of all the words in group 1) is $O(n^2f(n))$.

Words in group 2 which do not belong to group 1 contain factors of the form $a b^q a^p b^{q+1}$. Such a word is uniquely determined by $p<n$, $q \leq f(p)-1$ and the position of the first occurrence of $b^{q+1}$, which takes values between 0 and $n-q-1<n$. Hence the number of such words is also $O(n^2f(n))$.

An analogous counting argument applies to group 3.  Words in group 3 which have not yet been accounted for are uniquely determined by $p<n$ and the first position of
$a^{p+1}$, whence their number is $O(n^2)$.

Finally, for each word $v$ in group 4, we consider the first complete block $u=(a^{p}b^{q})^{g(p,q)}$ contained in $v$.  Then $v$ is uniquely determined by $p$, $q$ and the position of $u$ in $v$, hence the number of such words is again $O(n^2 f(n))$.

Thus, the complexity $p_x(n)=O(n^2f(n))$  as required. This completes the proof of Theorem~\ref{n2fn}. \end{proof}

%\noindent As consequence to Theorem~\ref{n2fn}:

\begin{corollary}\label{c:w3} For each non-decreasing integer function  $f(n)$ verifying $f(1)= 1$, $f(n)\leq n$ and $\lim_{n\rightarrow \infty}f(n)= +\infty$,  there exists an infinite word $x\in \{a,b\}^\nats$ of complexity $O(n^2 f(n))$ with $d_0(x)=d^*_1(x)=+\infty$.
\end{corollary}

\begin{proof}Let $x$ be as in Theorem~\ref{n2fn}. Due to result of the theorem, if $\Fac(x)\subseteq S^k$ for some language $S$, then $\displaystyle p^*_S(n)=\Omega(\sum_{p=1}^\frac{n-2}{2(2k-1)}f(p))$. Given any positive $M$, we can find $p_0$ such that $f(p_0)>M$; then, since $f(n)$ is non-decreasing,
\[\sum_{p=1}^\frac{n-2}{2(2k-1)}f(p)>\sum_{p=p_0}^\frac{n-2}{2(2k-1)}f(p)\geq M\left(\frac{n-2}{2(2k-1)}-p_0\right)>\frac{M}{4k}n+d\]
for an appropriate constant $d$ not depending on $n$. So, $p^*_S(n)$ grows faster than linearly. This means exactly that $d_1^*(x)=+\infty$; and $d_0(x)=+\infty$ due to Proposition \ref{t:l1l2}.
% of linear accumulative complexity. Fix a positive integer $M$ such that $p^*_S(n)\leq nM$ for all $n\geq 0$. But for all $n$ sufficiently large, we have
%\[p^*_S(n)\geq  \sum_{p=p_0}^\frac{n-2}{2(2k-1)} f(p).\]
%Since $f(p) \to \infty$ with $p$, we have
%\begin{eqnarray}\label{fp}p^*_S(n)\geq \sum_{p=p_0}^\frac{n-2}{2(2k-1)} f(p) > Mn\end{eqnarray}
%for some $n$ sufficiently large, a contradiction to the fact that $p^*_S(n)\leq Mn$ for all $n$.
\end{proof}

\begin{corollary} \label{c:alpha} For each $0<\alpha<1$, there exists an infinite word $x\in \{a,b\}^\nats$ of complexity $O(n^{2 +\alpha})$ such that  $c(x)\geq\alpha$.
\end{corollary}

\begin{proof} Fix $0<\alpha<1$.  Then applying Theorem~\ref{n2fn} to $f(n)=\lfloor n^\alpha \rfloor$, we have that there exists a word $x\in \{a,b\}^\nats$ of complexity $O(n^{2 +\alpha})$ such that if $\Fac(x)\subseteq S^k$ for some $S\subseteq \{a,b\}^*$ and $1\leq k<+\infty$, then \[p^*_S(n)=\Omega(n^{\alpha+1}).\]
Thus $c^*(x)\geq \alpha+1$, and hence $c(x)\geq \alpha$. \end{proof}

%\begin{corollary}For each $r\geq 3$, the inclusion $\W(r)\cap  \bigcup_{k\geq 1}\W_k \subseteq \W(r)$ is strict.
%\end{corollary}

\section{Non-factorial languages}
Positive results of previous sections concern mostly languages of factors of infinite words. In this section, we show that for a general non-factorial language, low complexity does not imply $d_0(L)<+\infty.$

\begin{theorem}\label{t:nonf}
 There exists a non-factorial language $L$ of complexity $p_L(n)=O(\log n)$ (and hence of cost zero) such that $d_0(L)=+\infty$.
\end{theorem}
\begin{proof}
 For each positive integer $n$, define $x_n \in \{0,1,2\}^*$ by $x_n=[n]_2 2$, where $[n]_2$ is the binary representation of $n$. For example, $x_2=102$ and $x_{65}=10000012$. Clearly, $|x_n|=\lfloor \log_2 n \rfloor+2$. Next define $y_n$ as the longest prefix of $x_n^{\omega}$ satisfying $|y_n|\log_2 |y_n| \leq n$. Thus for example $y_2=10$ since $2 \log_2 2 \leq 2 < 3 \log_2 3$ and $y_{65}=1000001210000012$ since $16 \log_2 16\leq 65 < 17 \log_2 17$. Finally, define $L=\{y_n|n\geq 1\}$.

We first claim that $|y_n|=\Theta(\frac{n}{\log n})$. Indeed, for $n \geq 2$, $|y_n|\geq 2$ so that $\log_2 |y_n|\geq 1$ and 
\begin{equation}\label{no1}|y_n|\leq \frac{n}{\log_2 |y_n|} \leq n.\end{equation} Since the length $|y_n|$ was chosen to be maximal, \begin{equation}\label{no2}|y_n|+1> \frac{n}{ \log_2 (|y_n|+1)} \geq \frac{n}{\log_2(n+1)},\end{equation} so $|y_n|=\Omega(n/\log n)$. Combining the (\ref{no1}) and (\ref{no2})  yields 
\[|y_n|\leq \frac{n}{\log_2(\frac{n}{\log_2(n+1)}-1)}.\] Since $\frac{n}{\log_2(n+1)}-1$ is asymptotically equivalent to $\frac{n}{\log_2 n}$ we deduce $|y_n|=O(\frac{n}{\log n})$. Together with the lower bound above, this gives $|y_n|=\Theta(\frac{n}{\log n})$ as required.

Next we show that $d_0(L)=+\infty$. Indeed, suppose by contrary that $L\subseteq S^k$ for some $k \in \mathbb Z$ and some set $S$ of bounded complexity. Since \[\frac{|y_n|}{|x_n|}=\Theta\left(\frac{n}{(\log_2 n)^2}\right),\] there exists an integer $n_0>0$ such that for all $n>n_0$, we have $|y_n|\geq (k+1)|x_n|$. This means that for all $n>n_0$, the word $y_n$ contains at least $k+1$ occurrences of $2$, and at least two of them are located in the same word from $S$, denote it by $s_n$. Since between two occurrences of $2$ in $s_n$, there is exactly the binary representation of $n$, all $s_n$ for $n \geq n_0$ are pairwise distinct.

Now for each $n \geq n_0$ consider the set $S(n)=\{s_m|n_0<m\leq n\}\subseteq S$. It contains $n-n_0$ distinct words, and the length of each of them is $o(n)$: indeed, $|s_m|\leq |y_m|=\Theta(n /\log n)$. So the accumulative complexity of $S$ grows faster than linearly, which is impossible if its usual complexity is bounded.

It remains to prove that $p_L(n)=\Theta(\log n)$. Indeed, \[p_L(n)=\#\{m:  |y_m|=n\}.\]In other words,
\[p_L(n)=\#\{m:  n\log_2 n \leq m < (n+1)\log_2 (n+1)\}.\] Whence,
\begin{eqnarray*}
p_L(n)&=&\lceil (n+1)\log_2(n+1)\rceil - \lceil n \log_2 n \rceil = \Theta(\log n).
\end{eqnarray*}
This completes the proof of Theorem~\ref{t:nonf}.
\end{proof}

The language $L$ in Theorem~\ref{t:nonf} provides an example of a language of cost equal to $0$ and having infinite cost dimension. We do not know whether there exists an infinite word $x$ with $c(x)=0$ and $d_0(x)=+\infty.$

\end{document}